\documentclass[final,leqno,onefignum,onetabnum]{siamltex1213}

\usepackage{amsmath, amsfonts, amssymb, cite, graphicx, xspace}

\newtheorem{assumption}{Assumption}
\newtheorem{algorithm}{Algorithm}
\newtheorem{remark}{Remark}
\newtheorem{example}{Example}

\newenvironment{algorithminit}[1]{\ \\{\em Initialization}: #1\begin{list}{\labelenumi}{\topsep0in\itemsep0in\parsep0in\labelwidth1in\usecounter{enumi}}}{\setcounter{enumii}{\value{enumi}}\end{list}}
\newenvironment{algorithmoper}[1]{{\em Operation}: #1\begin{list}{\labelenumi}{\topsep0in\itemsep0in\parsep0in\labelwidth1in\usecounter{enumi}\setcounter{enumi}{\value{enumii}}}}{\hfill$\blacksquare$\end{list}}

\title{Convergence analysis of approximate primal solutions in dual first-order methods\footnotemark[4] \thanks{This work has been sponsored in part by the Swedish Research Council and the Swedish Foundation for Strategic Research.}} 

\author{Jie Lu\footnotemark[2]\ 
\and Mikael Johansson\footnotemark[3]}

\begin{document}
\maketitle
\renewcommand{\thefootnote}{\fnsymbol{footnote}}
\footnotetext[2]{Department of Signals and Systems, Chalmers University of Technology, SE-412 96 Gothenburg, Sweden.(\email{jielu@chalmers.se})}
\footnotetext[3]{Department of Automatic Control, ACCESS Linnaeus Center, School of Electrical Engineering, KTH Royal Institute of Technology, SE-100 44 Stockholm, Sweden.(\email{mikaelj@kth.se})}
\footnotetext[4]{A shorter, conference version of this paper has been published in the Proceedings of IEEE Conference on Decision and Control, pp.6861-6867, Florence, Italy, 2013.}


\begin{abstract}
Dual first-order methods are powerful techniques for large-scale convex optimization. Although an extensive research effort has been devoted to studying their convergence properties, explicit convergence rates for the primal iterates have only been established under global Lipschitz continuity of the dual gradient. This is a rather restrictive assumption that does not hold for several important classes of problems. In this paper, we demonstrate that primal convergence rate guarantees can also be obtained when the dual gradient is only locally Lipschitz. The class of problems that we analyze admits general convex constraints including nonlinear inequality, linear equality, and set constraints. As an approximate primal solution, we take the minimizer of the Lagrangian, computed when evaluating the dual gradient. We derive error bounds for this approximate primal solution in terms of the errors of the dual variables, and establish convergence rates of the dual variables when the dual problem is solved using a projected gradient or fast gradient method. By combining these results, we show that the suboptimality and infeasibility of the approximate primal solution at iteration $k$ are no worse than $O(1/\sqrt{k})$ when the dual problem is solved using a projected gradient method, and $O(1/k)$ when a fast dual gradient method is used.

\end{abstract}

\begin{keywords}dual optimization, \and first-order methods, \and primal convergence\end{keywords}

\begin{AMS}\end{AMS}

\pagestyle{myheadings}
\thispagestyle{plain}

\section{Introduction}
Lagrangian duality is a widely-used approach in large-scale optimization, especially when there are a few constraints that complicate an otherwise simple problem \cite{Lasdon70, Bertsekas99}. Although many first-order methods can be applied to solve such problems directly in the primal space, the iteration cost can be very high since the projection onto the constraint set is often computationally difficult \cite{Nesterov04}. The corresponding dual problem has a more desirable structure: the dual constraint set has a simple form and the (sub)gradient of the dual function is relatively easy to evaluate. In addition, the dual function is often additive and suitable for distributed implementation, which has been exploited in a wide range of recent applications, including communication systems \cite{Low99, XiaoL04b}, large-scale control \cite{Giselsson13}, and multi-agent systems \cite{Nedic10b}. 

There are many practical and theoretical subtleties in using dual optimization methods to generate optimal solutions to the engineering problems cited above. First, one needs to ensure that the dual optimal value agrees with the primal optimal value (\emph{i.e.}, that there is no duality gap). For convex optimization problems, this can be done by verifying Slater's constraint qualifications \cite{Bertsekas99}. 
Then, one typically needs to guarantee that the iterates generated by the dual optimization method converge to a dual optimum, which is not always true. For instance, the subgradient method with constant step-size achieves suboptimality only. Further, for most applications it is desirable to construct approximate primal solutions (representing the actual decisions to implement) from the dual iterates. Whether the approximate primal solutions converge to a primal optimal solution or not is often of great practical concern. Moreover, to be able to assess solution times and understand how they depend on problem data, it is preferable to estimate how quickly the solution converges. This motivates research on on-line construction of approximate primal solutions and studying their convergence properties.

A number of results on the convergence properties of approximate primal solutions have been reported in the literature~\cite{Larsson99,Nedic09b,Devolder12,Giselsson13,Patrinos13,Beck14,Necoara14,Nedelcu14}. At one extreme are results on non-smooth convex problems with nonlinear constraints, \emph{e.g.}~\cite{Nedic09b,Larsson99}, where the corresponding dual function is also non-smooth in general. For such problems, one typically applies the subgradient method to the dual problem and forms running averages of the generated primal iterates to construct an approximate primal solution. Such an approximate primal solution converges asymptotically to the primal optimal set with diminishing step-sizes \cite{Larsson99} and has guaranteed bounds on suboptimality and infeasibility when a constant step-size is used \cite{Nedic09b}. At the other extreme are problems for which the dual function is differentiable and has globally Lipschitz continuous gradient over the entire dual feasible set, \emph{e.g.}~\cite{Devolder12,Giselsson13,Patrinos13,Beck14,Necoara14,Nedelcu14}. To ensure differentiability of the dual function, one often needs to assume strong convexity of the objective function \cite{Giselsson13,Patrinos13,Beck14,Necoara14} or approach the dual problem using an augmented Lagrangian \cite{Devolder12, Nedelcu14}. To make the dual gradient globally Lipschitz, the references cited above typically require the inequality and equality constraints to be linear. One exception is~\cite{Necoara14} that allows for nonlinear inequality constraints, but not equality constraints. However, \cite{Necoara14} assumes that both the objective and the inequality constraint functions are twice differentiable and that the Jacobian of the constraint functions is element-wise bounded. The globally Lipschitz dual gradient not only simplifies analysis but also allows the application of dual gradient and fast gradient methods (\emph{e.g.}, \cite{Nesterov04,Tseng08,Beck09}) that achieve sublinear convergence rates for the dual iterates. This leads to sublinear convergence rates of the approximate primal solution, be it either the primal iterates \cite{Devolder12,Giselsson13,Beck14} or their running average \cite{Patrinos13,Necoara14,Nedelcu14}.

In this paper, we consider a general class of convex optimization problems that covers the less explored middle ground between these two extremes. In particular, we focus on a class of convex optimization problems with a strongly convex but not necessarily differentiable objective function. We allow the problems to have all three types of convex constraints: nonlinear inequalities, linear equalities, and set constraints, while the related references \cite{Larsson99,Nedic09b,Devolder12,Giselsson13,Patrinos13,Beck14,Necoara14,Nedelcu14} tackle problems in the absence of either nonlinear constraints or equality constraints. This problem class leads to a differentiable dual function with \emph{locally} Lipschitz gradient on the dual feasible set and generalizes the problems with \emph{globally} Lipschitz dual gradient considered in \cite{Giselsson13,Patrinos13,Beck14,Necoara14}.

For this problem class, we consider the unique minimizer of the Lagrangian for given dual variables as an approximate primal solution and relate the errors of this approximate primal solution in primal optimality and feasibility to those of the dual variables in dual optimality. Based on such relationships, we study the convergence properties of the approximate primal solution when the dual variables are generated from the application of the classical projected gradient and fast gradient methods to the dual problem. Specifically, by imposing mild assumptions on the smoothness of the inequality constraint functions, we construct a sufficient condition on the step-size to guarantee convergence of the dual iterates generated by the projected dual gradient method and prove that they converge sublinearly at a rate of order $O(1/k)$. It is worthwhile to mention that this is a new result, as the existing results on the $O(1/k)$ convergence rate of the projected gradient method are established under global Lipschitz continuity of the objective gradient, while for our problem, the dual gradient is only {\em locally} Lipschitz on the dual feasible set. This leads to one of our main results, which states that the primal iterates (\emph{i.e.}, our approximate primal solution at each iteration) converge to optimality and feasibility at a rate no worse than $O(1/\sqrt{k})$. 
By assuming boundedness of the subgradients of the inequality constraint functions, we show that the fast gradient methods in \cite{Tseng08,Beck09} can be applied to solve the dual problem and guarantee the $O(1/k^2)$ convergence rate of the dual iterates. As a result, the convergence rates of the primal iterates in both optimality and feasibility are improved to $O(1/k)$. 

The paper is organized as follows: Section~\ref{sec:probform} gives a formal problem statement, while Section~\ref{sec:PEB} establishes bounds for the error of the approximate primal solution in terms of errors of the dual variables. Convergence rate bounds for the dual and primal iterates in several dual first-order methods are derived in Section~\ref{sec:CPI}. Section~\ref{sec:numeexam} uses simulations to compare the practical performance of different choices of approximate primal solutions in various dual first-order methods. 
Finally, Section~\ref{sec:concl} concludes the paper. The proofs are in the appendix. 

\subsection{Notation}

The following notation is adopted throughout the paper: Let $\mathbb{R}^n_+$ and $\mathbb{R}^n_-$ be the set of nonnegative and negative vectors in $\mathbb{R}^n$, respectively. For a vector $x\in\mathbb{R}^n$, let $x^{(i)}\in\mathbb{R}$, $i=1,2,\ldots,n$ denote the $i$th element of $x$ and $x^{(i:j)}\in\mathbb{R}^{j-i+1}$, $1\le i\le j\le n$ the vector consisting of the $i$th, $(i+1)$th, $\ldots$, $j$th elements of $x$. In addition, for any $x,y\in\mathbb{R}^n$, let $\max\{x,y\}$ be the element-wise maximum operation, \emph{i.e.}, $(\max\{x,y\})^{(i)}=\max\{x^{(i)},y^{(i)}\}$ $\forall i\in\{1,\ldots,n\}$. We use $\|\cdot\|$, $\|\cdot\|_1$, $\|\cdot\|_\infty$, and $\|\cdot\|_F$ to represent the Euclidean, $\ell_1$, infinity, and Frobenius norm, respectively. For any matrix $A\in\mathbb{R}^{p\times n}$, let $\sigma_{\max}(A)=\sqrt{\lambda_{\max}(A^T A)}$ and $\sigma_{\min}(A)=\max\{\sqrt{\lambda_{\min}(AA^T)},\sqrt{\lambda_{\min}(A^T A)}\}$, where $\lambda_{\max}(\cdot)$ and $\lambda_{\min}(\cdot)$ represent the largest and smallest eigenvalues of a real symmetric matrix. We allow $A$ to have zero dimension, \emph{i.e.}, $p=0$ or $n=0$, in which cases we let $\sigma_{\max}(A)=0$. 
For any function $h:\mathbb{R}^n\rightarrow\mathbb{R}$, let $\partial h(x)\subset\mathbb{R}^n$ be its subdifferential at $x\in\mathbb{R}^n$. If $h$ is differentiable at $x$, then $\partial h(x)=\{\nabla h(x)\}$, where $\nabla h(x)\in\mathbb{R}^n$ is the gradient of $h$ at $x$ and its $i$th element is represented by $\nabla^{(i)}h(x)$. For any set $Q\subseteq\mathbb{R}^n$, let $\operatorname{relint}Q$ be its relative interior, $\operatorname{conv}Q$ its convex hull, $\operatorname{diam}(Q)$ its diameter, $|Q|$ its cardinality, and $\mathcal{P}_Q[\cdot]$ the projection onto $Q$. 

\section{Problem formulation}\label{sec:probform}

We consider the following optimization problem with inequality, equality, and set constraints:
\begin{align}
\begin{array}{ll}\underset{x\in\mathbb{R}^n}{\mbox{minimize}} & f(x)\\ \operatorname{subject\,to} & g^{(i)}(x)\le0,\quad i=1,2,\ldots,m,\\ &Ax+b=0,\\ &x\in X.\end{array}\label{eq:generalprimalprob}
\end{align}
Here, $f:\mathbb{R}^n\rightarrow\mathbb{R}$ is the objective function, $g^{(i)}:\mathbb{R}^n\rightarrow\mathbb{R}$, $\forall i\in\{1,2,\ldots,m\}$ represent the nonlinear inequality constraint functions, $A\in\mathbb{R}^{p\times n}$ and $b\in\mathbb{R}^p$ encode the linear equality constraints, and $X\subseteq\mathbb{R}^n$ is a closed and convex set. In addition, let the following assumption hold:

\begin{assumption}\label{asm:generalprimalprob}
Problem~\eqref{eq:generalprimalprob} satisfies the following:
\begin{enumerate}
\item[(a)] The objective function $f$ is strongly convex over $X$ with convexity parameter $\theta>0$, i.e., $f(y)-f(x)-\tilde{\nabla} f(x)^T(y-x)\ge\frac{\theta}{2}\|x-y\|^2$, $\forall x,y\in X$, $\forall\tilde{\nabla}f(x)\in
\partial f(x)$.\footnote{$f$ is not necessarily differentiable. For instance, $f$ could be a quadratic function plus an $\ell_1$ norm.}
\item[(b)] Each $g^{(i)}$, $i\in\{1,2,\ldots,m\}$ is convex over $X$ and satisfies a Lipschitz condition on $X$: $\|g^{(i)}(x)-g^{(i)}(y)\|\le L_i\|x-y\|$ $\forall x,y\in X$ for some $L_i>0$.
\item[(c)] There exists $\tilde{x}\in\operatorname{relint} X$ such that $g^{(i)}(\tilde{x})<0$ $\forall i\in\{1, 2, \ldots, m\}$ and $A\tilde{x}+b=0$.
\item[(d)] The number of inequality and equality constraints is not zero, \emph{i.e.}, $m+p\neq0$. If $p\neq0$, then $A$ is not a zero matrix.
\end{enumerate}
\end{assumption}

Assumptions~\ref{asm:generalprimalprob}(a),~\ref{asm:generalprimalprob}(b), and~\ref{asm:generalprimalprob}(c) guarantee that there is a unique optimal solution $x^\star$ to problem~\eqref{eq:generalprimalprob} and that the optimal value $f^\star=f(x^\star)$ is finite. In addition, they ensure that problem~\eqref{eq:generalprimalprob} has no duality gap when dualizing the inequality and equality constraints, \emph{i.e.}, $f^\star$ is equal to the optimal value $d^\star$ of the corresponding dual problem, and that the dual optimal set $D^\star$ is nonempty \cite[Prop. 5.3.2]{Bertsekas99}. Note that we only require the convexity and Lipschitz continuity in Assumptions~\ref{asm:generalprimalprob}(a) and~\ref{asm:generalprimalprob}(b) to hold over $X$, and not globally over the entire $\mathbb{R}^n$.

To formulate the dual problem of \eqref{eq:generalprimalprob}, we first introduce the Lagrangian function $\mathcal{L}:\mathbb{R}^n\times\mathbb{R}^{m+p}\rightarrow\mathbb{R}$ associated with~\eqref{eq:generalprimalprob}:
\begin{align*}
\mathcal{L}(x,u)=f(x)+\sum_{i=1}^mu^{(i)}g^{(i)}(x)+\Bigl(u^{(m+1:m+p)}\Bigr)^T(Ax+b).
\end{align*}
Given the Lagrangian $\mathcal{L}$, the dual function $d:\mathbb{R}^{m+p}\rightarrow\mathbb{R}$ can be expressed as
\begin{align}
d(u)&=\min_{x\in X}\mathcal{L}(x,u)\nonumber\\
&=f(\bar{x}(u))+\sum_{i=1}^mu^{(i)}g^{(i)}(\bar{x}(u))+\Bigl(u^{(m+1:m+p)}\Bigr)^T(A\bar{x}(u)+b),\label{eq:generaldualfunction}
\intertext{where}
\bar{x}(u)&\in\operatorname{arg\;min}_{x\in X}\mathcal{L}(x,u),\nonumber
\end{align}
and the Lagrange dual problem of~\eqref{eq:generalprimalprob} is
\begin{align}
\begin{array}{ll}\underset{u\in\mathbb{R}^{m+p}}{\mbox{maximize}} & d(u)\\ \operatorname{subject\,to} & u\in D\triangleq\{u\in\mathbb{R}^{m+p}:u^{(1:m)}\in\mathbb{R}^m_+\}.
\end{array}\label{eq:generaldualprob}
\end{align}
Since $d$ is concave, the dual problem (\ref{eq:generaldualprob}) is a convex optimization problem. Moreover, for every $u\in D$, $\mathcal{L}(\cdot,u)$ is strongly convex over $X$, so $\bar{x}(u)$ exists and is unique. Furthermore, $\bar{x}(u)=x^\star$ if $u=u^\star$ for some dual optimal solution $u^\star\in D^\star$ \cite[Prop. 6.1.1]{Bertsekas03}. This makes $\bar{x}(u)$ a legitimate candidate for an \textit{approximate primal solution} to~\eqref{eq:generalprimalprob} based on the dual variable $u\in D$. 

Next, we establish the boundedness of $\bar{x}(u)$ and then the differentiability of $d$:

\begin{lemma}\label{lem:xbounded}
Consider problem~\eqref{eq:generalprimalprob} under Assumption~\ref{asm:generalprimalprob}. Then, for any compact set $S\subset D$, the set $\{\bar{x}(u):u\in S\}$ is bounded.
\end{lemma}

\begin{proof}
See Appendix~\ref{ssec:prooflemxbounded}.
\end{proof}

With Lemma~\ref{lem:xbounded} and Danskin's Theorem \cite{Bertsekas99}, it can be shown that the dual function $d$ is differentiable at every point in $D$. Moreover, for any $u\in D$,
\begin{align}
\nabla d(u)=[g^{(1)}(\bar{x}(u)),\ldots,g^{(m)}(\bar{x}(u)),(A\bar{x}(u)+b)^T]^T.\label{eq:generalgradd}
\end{align}

\begin{remark}Equations~\eqref{eq:generaldualfunction} and~\eqref{eq:generalgradd} suggest that the primal function value at $\bar{x}(u)$, \emph{i.e.,} $f(\bar{x}(u))$, can be expressed in terms of the dual variable $u\in D$ and the dual function $d$ as follows:
\begin{align}
f(\bar{x}(u))=d(u)-\nabla d(u)^Tu,\quad\forall u\in D.\label{eq:f=d-ndu}
\end{align}
This relationship will be essential in the results that we derive shortly.
\end{remark}

In the above setting, the goal of this paper is to (a) quantify how close the approximate primal solution $\bar{x}(u)$ is to optimality and feasibility for any given dual feasible point $u\in D$, and (b) to derive the convergence \emph{rate} of $\bar{x}(u)$ when the dual problem~\eqref{eq:generaldualprob} is solved using some common first-order methods. We investigate optimality in terms of both the distance to the optimizer
\begin{align*}
	\Vert \bar{x}(u)-x^{\star} \Vert
\end{align*}
and the error in primal objective value 
\begin{align*}
	\vert f(\bar{x}(u))-f^{\star}\vert
\end{align*}
while primal infeasibility is captured by the quantity
\begin{align*}
\Delta(\bar{x}(u))=\Bigl(\|A\bar{x}(u)+b\|^2+\sum_{i=1}^m\bigl(\max\{0,g^{(i)}(\bar{x}(u))\}\bigr)^2\Bigr)^{1/2}.
\end{align*}

\subsection{Comparison with related work}

It is instructive to compare Problem~\eqref{eq:generalprimalprob} with the problem classes considered in the related works \cite{Larsson99,Nedic09b,Devolder12,Giselsson13,Patrinos13,Beck14,Necoara14,Nedelcu14} that also study primal convergence in dual first-order methods.

First of all, note that~\eqref{eq:generalprimalprob} allows for all three types of standard convex constraints (convex inequality, linear equality, and convex set constraints), while \cite{Larsson99,Nedic09b,Devolder12,Giselsson13,Patrinos13,Beck14,Necoara14,Nedelcu14} do not. The constraints in \cite{Devolder12,Giselsson13,Patrinos13,Beck14,Nedelcu14} must be linear, and although \cite{Larsson99,Nedic09b,Necoara14} consider nonlinear inequality constraints, they do not allow for linear equality constraints.

Like our Assumption~\ref{asm:generalprimalprob}, references \cite{Giselsson13,Patrinos13,Beck14,Necoara14} also assume strong convexity of the objective function $f$. Clearly, problem~\eqref{eq:generalprimalprob} generalizes the linearly constrained problems considered in \cite{Giselsson13,Patrinos13,Beck14}. In addition, the problem class with nonlinear inequality constraints in \cite{Necoara14} requires that the objective and the inequality constraint functions are twice differentiable and that the Jacobian of the inequality constraint functions is element-wise bounded. These are more restrictive than Assumption~\ref{asm:generalprimalprob}.

Strong convexity of the objective function is relaxed to convexity in \cite{Larsson99,Nedic09b,Devolder12,Nedelcu14}. In \cite{Larsson99,Nedic09b}, the dual function is non-differentiable and therefore only the subgradient method can be applied to the dual, which explains the lack of convergence rate gurantees. In \cite{Devolder12,Nedelcu14}, quadratic augmented Lagrangians are used to obtain a differentiable dual function. Nevertheless, they still require the constraint set $X$ to be compact.

\section{Primal errors in optimality and feasibility}\label{sec:PEB}

In this section, we bound the errors of the approximate primal solution $\bar{x}(u)\in X$ in optimality and feasibility in terms of the errors of the dual variable $u\in D$. To present our first result, we introduce the following notation: for any $u\in D$, let
\begin{align}
\gamma(u)&=\frac{\sqrt{m+1}}{\theta}\max\Bigl\{\sigma_{\max}(A),\sup_{q\in G(u)}\|q\|\Bigr\},\label{eq:gammau}
\intertext{where}
G(u)&=\bigcup_{i=1}^m\partial g^{(i)}(\bar{x}(u))\subset\mathbb{R}^n.\label{eq:Gu}
\end{align}
Since $G(u)$ is a compact set \cite[Prop. 4.2.1]{Bertsekas03}, $0\le\gamma(u)<\infty$. Also, if $\gamma(u')=0$ for some $u'\in D$, then $\bar{x}(u)=x^\star$ $\forall u\in D$.\footnote{To see this, note that $\gamma(u')=0$ implies $p=0$ and $\partial g^{(i)}(\bar{x}(u'))=\{0\}$ $\forall i\in\{1,\ldots,m\}$. Hence, $\partial_x\mathcal{L}(\bar{x}(u'),u)=\partial f(\bar{x}(u'))$ $\forall u\in D$, where $\partial_x\mathcal{L}$ represents the subdifferential of $\mathcal{L}$ with respect to the first argument. Since $\bar{x}(u')$ minimizes $\mathcal{L}(\cdot,u')$ over $X$, there exists $\tilde{\nabla}f(\bar{x}(u'))\in\partial f(\bar{x}(u'))$ such that $\tilde{\nabla}f(\bar{x}(u'))^T(x-\bar{x}(u'))\ge0$. Therefore, $\bar{x}(u')=\operatorname{arg\;min}_{x\in X}\mathcal{L}(x,u)=\bar{x}(u)$ $\forall u\in D$ and thus $\bar{x}(u)=x^\star$ $\forall u\in D$.} This means that the primal optimal solution $x^\star$ can be simply found by arbitrarily picking $u\in D$ and computing $\bar{x}(u)$. Hence, in the rest of the paper, we exclude this trivial case and assume $\gamma(u)>0$ $\forall u\in D$.

Then, consider the following lemma:

\begin{lemma}\label{lem:xx<=cuu}
Consider problem~\eqref{eq:generalprimalprob} under Assumption~\ref{asm:generalprimalprob}. For any $u,v\in D$,
\begin{align}
\|\bar{x}(u)-\bar{x}(v)\|\le\min\{\gamma(u),\gamma(v)\}\|u-v\|,\label{eq:xx<=uuloose}
\end{align}
where $\gamma(u),\gamma(v)\in(0,\infty)$ are defined in \eqref{eq:gammau}.
\end{lemma}

\begin{proof}
See Appendix~\ref{ssec:prooflemxx<=cuu}.
\end{proof}

Lemma~\ref{lem:xx<=cuu} allows one to relate the primal error $\|\bar{x}(u)-x^\star\|$ to the dual error $\|u-u^\star\|$ for any $u\in D$ and any $u^\star\in D^\star$. In addition, the next theorem bounds $\|\bar{x}(u)-x^\star\|$ by virtue of the error $d^\star-d(u)$ in dual optimality.

\begin{theorem}\label{thm:xx<=cuu}
Consider problem~\eqref{eq:generalprimalprob} under Assumption~\ref{asm:generalprimalprob}. For any $u\in D$ and any $u^\star\in D^\star$, 
\begin{align}
\|\bar{x}(u)-x^\star\|&\le\gamma(u^\star)\|u-u^\star\|,\label{eq:xx<=gammauu}\displaybreak[0]\\
\|\bar{x}(u)-x^\star\|&\le\sqrt{\frac{2(d^\star-d(u))}{\theta}},\quad \label{eq:xx<=sqrtddtheta}
\end{align}
where $\gamma(u^\star)\in (0,\infty)$ is defined in \eqref{eq:gammau}. 
\end{theorem}

\begin{proof}
See Appendix~\ref{ssec:proofthmxx<=cuu}.
\end{proof}

Note that both Lemma~\ref{lem:xx<=cuu} and Theorem~\ref{thm:xx<=cuu} do not require the Lipschitz condition in Assumption~\ref{asm:generalprimalprob}(b), which, however, is needed for deriving other results below.

Having derived bounds on $\|\bar{x}(u)-x^\star\|$, we turn our attention to the primal error $|f(\bar{x}(u))-f^\star|$. To this end, for any compact subset $S\subset D$, define
\begin{align}
L(S)&=\sup_{u\in S}\gamma(u)\Bigl(\sigma_{\max}^2(A)+\sum_{i=1}^mL_i^2\Bigr)^{1/2}>0.\label{eq:LD'}
\end{align}
From Lemma~\ref{lem:xbounded} and \cite[Prop. 4.2.3]{Bertsekas03}, the boundedness of $S$ implies that the set 
$\cup_{u\in S}G(u)$ is bounded, so $L(S)<\infty$. Next, we show that $L(S)$ is a Lipschitz constant of $\nabla d$ on the compact set $S\subset D$:

\begin{proposition}\label{pro:gradientdualLipschitz}
Consider problem~\eqref{eq:generalprimalprob} under Assumption~\ref{asm:generalprimalprob}. Then, on every compact set $S\subset D$, $\nabla d$ satisfies a Lipschitz condition:
\begin{align}
\|\nabla d(u)-\nabla d(v)\|\le L(S)\|u-v\|,\quad\forall u, v\in S,\label{eq:gdgd<=Luu}
\end{align}
where $L(S)\in(0,\infty)$ is defined in \eqref{eq:LD'}.
Moreover, if $S$ is convex,
\begin{align}
d(v)-d(u)-\nabla d(u)^T(v-u)\ge&-\frac{L(S)}{2}\|u-v\|^2,\quad\forall u, v\in S. \label{eq:ddgduu>=-L2uu}
\end{align}
\end{proposition}

\begin{proof}
See Appendix~\ref{ssec:proofprogradientdualLipschitz}.
\end{proof}

The {\em local} Lipschitz continuity of $\nabla d$ on $D$ established in Proposition~\ref{pro:gradientdualLipschitz} allows to guarantee bounds on the primal error $|f(\bar{x}(u))-f^\star|$ and the primal infeasibility $\Delta(\bar{x}(u))$ of any $\bar{x}(u)$ with $u$ in some compact set $S\subset D$. The basic idea for deriving such bounds is to use \eqref{eq:f=d-ndu}, which gives
\begin{align}
f(\bar{x}(u))-f^\star=-\nabla d(u)^Tu+d(u)-d^\star\label{eq:ff<=-gdu}
\end{align}
and then bound $-\nabla d(u)^Tu$ using \eqref{eq:ddgduu>=-L2uu} with $v=u+\frac{1}{L(S)}\nabla d(u)$. However, since such a $v$ may not belong to $S$, we introduce the set
\begin{align}
\Phi(S)&=\operatorname{conv}\Bigl(\Bigl\{\mathcal{P}_D[u+\beta\nabla d(u)]:u\in S,\; \beta\in [0, 1/L(S)]\Bigr\}\Bigr),\label{eq:D''}
\end{align}
which is compact and convex. In addition, $S\subseteq\Phi(S)\subset D$. Hence, if $u\in S$ and we let $v=\mathcal{P}_D[u+\frac{1}{L(\Phi(S))}\nabla d(u)]$, then $u,v\in\Phi(S)$ and we can apply \eqref{eq:ddgduu>=-L2uu} over $\Phi(S)$. 
The following theorem provides the formal results:
\begin{theorem}\label{thm:primalfunctionvalue}
Consider problem~\eqref{eq:generalprimalprob} under Assumption~\ref{asm:generalprimalprob}. Let $S\subset D$ be compact. 
Then, for any $u\in S$ and any $u^\star\in D^\star$,
\begin{align}
f(\bar{x}(u))-f^\star&\le\Bigl(\|u\|_\infty\sqrt{2L(\Phi(S))(m+p)}+\sqrt{d^\star-d(u)}\Bigr)\sqrt{d^\star-d(u)},\label{eq:ff<=usqrt2Lmp2sqrtdddd}\displaybreak[0]\\
f(\bar{x}(u))-f^\star&\ge-\|u^\star\|\sqrt{2L(\Phi(S))(d^\star-d(u))},\label{eq:ff>=-usqrt2Ldd}\displaybreak[0]\\
\Delta(\bar{x}(u))&\le\sqrt{2L(\Phi(S))(d^\star-d(u))},\label{eq:Delta<=2Ldd}
\end{align}
where $\Phi(S)$ and $L(\Phi(S))\in(0,\infty)$ are defined in \eqref{eq:D''} and \eqref{eq:LD'}.
\end{theorem}

\begin{proof}
See Appendix~\ref{ssec:proofthmprimalfunctionvalue}.
\end{proof}

The bounds provided in \eqref{eq:ff<=usqrt2Lmp2sqrtdddd}, \eqref{eq:ff>=-usqrt2Ldd}, and \eqref{eq:Delta<=2Ldd} depend on the compact set $\Phi(S)$ defined in \eqref{eq:D''}.
Thus, unlike Theorem~\ref{thm:xx<=cuu}, the results in Theorem~\ref{thm:primalfunctionvalue} only hold {\em locally}, which stems from the fact that $\nabla d$ is locally Lipschitz continuous on $D$. Nevertheless, under the assumption below, similar conclusions can be established {\em globally} over $D$:

\begin{assumption}\label{asm:boundedsubgrad}
The set $\cup_{u\in D}G(u)$ is bounded.
\end{assumption}

Assumption~\ref{asm:boundedsubgrad} can be satisfied when each constraint function $g^{(i)}$, $i=1,2,\ldots,m$ is affine or the constraint set $X$ is compact (cf. \cite[Prop. 4.2.3]{Bertsekas03}). For another example, if each $g^{(i)}$ is differentiable at every point of $X$ and satisfies the Lipschitz condition in Assumption~\ref{asm:generalprimalprob}(b) on an open set containing $X$, then $\|\nabla g^{(i)}(x)\|\le L_i$ $\forall x\in X$, which implies that Assumption~\ref{asm:boundedsubgrad} holds. However, if the Lipschitz condition only holds on $X$ as in Assumption~\ref{asm:generalprimalprob}(b), then $\|\nabla g^{(i)}(x)\|$ may be unbounded on $X$\footnote{For instance, let $X=\{x\in\mathbb{R}^2:x^{(1)}\ge1,\;x^{(2)}=1\}$. Also, let $g^{(i)}:\mathbb{R}_+^2\rightarrow\mathbb{R}$ be defined as $g^{(i)}(x)=-\bigl(x^{(1)}\bigr)^{x^{(2)}}$, which is differentiable, is convex, and satisfies a Lipschitz condition on $X$. However, $\|\nabla g^{(i)}(x)\|^2=1+\bigl(x^{(1)}\ln x^{(1)}\bigr)^2$ $\forall x\in X$, which is unbounded.} and Assumption~\ref{asm:boundedsubgrad} is thus not guaranteed.

\begin{remark}
Note that even when both Assumption~\ref{asm:generalprimalprob} and Assumption~\ref{asm:boundedsubgrad} are imposed, our results generalize those in references~\cite{Giselsson13,Patrinos13,Beck14,Necoara14}, since we do not require $f$ and $g^{(i)}$, $i=1,2,\ldots,m$ to be differentiable. Also note that Assumption~\ref{asm:boundedsubgrad} is not universally imposed throughtout the paper; most results hold without Assumption~\ref{asm:boundedsubgrad}.
\end{remark}

Under Assumption~\ref{asm:boundedsubgrad}, we have $\sup_{u\in D}\gamma(u)<\infty$, which leads to the Lipschitz continuity of $\nabla d$ on $D$ and the following error bounds:

\begin{corollary}\label{cor:primalfunctionvalueboundedsubgrad}
Consider problem~\eqref{eq:generalprimalprob} under Assumptions~\ref{asm:generalprimalprob} and~\ref{asm:boundedsubgrad}. Then, $\nabla d$ satisfies a Lipschitz condition on $D$: $\|\nabla d(u)-\nabla d(v)\|\le \tilde{L}\|u-v\|$ $\forall u,v\in D$, where
\begin{align*}
\tilde{L}=\sup_{u\in D}\gamma(u)\Bigl(\sigma_{\max}^2(A)+\sum_{i=1}^mL_i^2\Bigr)^{1/2}\in(0,\infty).
\end{align*}
Moreover, for any $u\in D$, \eqref{eq:ff<=usqrt2Lmp2sqrtdddd}--\eqref{eq:Delta<=2Ldd} hold with $L(\Phi(S))$ replaced by $\tilde{L}$.
\end{corollary}

\begin{proof}
See Appendix~\ref{ssec:proofcorprimalfunctionvalueboundedsubgrad}
\end{proof}

In the final part of this section, we study a special case of~\eqref{eq:generalprimalprob} where all the constraints are linear and derive sharper and more explicit primal error bounds. Specifically, we consider
\begin{align}
\begin{array}{ll}\underset{x\in \mathbb{R}^{n}}{\mbox{minimize}} & f(x)\\ \operatorname{subject\,to} & A'x+b'\le0,\\ &Ax+b=0,\\ & x\in X,\end{array}\label{eq:linearprimalprob}
\end{align}
where $A'\in\mathbb{R}^{m\times n}$, $b'\in\mathbb{R}^m$, and $\le$ represents element-wise inequality. For convenience, let $\tilde{A}=[(A')^T, A^T]^T\in\mathbb{R}^{(m+p)\times n}$ and $\tilde{b}=[(b')^T, b^T]^T\in\mathbb{R}^{m+p}$. Without loss of generality, we assume $\tilde{A}$ is not a zero matrix.

If $f$ is strongly convex over the whole $\mathbb{R}^n$, $X$ is a polyhedral set, and the constraint set of problem~\eqref{eq:linearprimalprob} is nonempty, then Assumption~\ref{asm:generalprimalprob}(c) can be removed \cite[Prop. 5.2.1]{Bertsekas99}. Also, Assumption~\ref{asm:boundedsubgrad} is automatically satisfied for this problem due to the linearity of the constraints. Besides, $\bar{x}(u)$ exists and is unique for any $u\in\mathbb{R}^{m+p}$ and $d$ is differentiable over $\mathbb{R}^{m+p}$.

Following the proof of Lemma~\ref{lem:xx<=cuu}, we show in the corollary below that the distance between approximate primal solutions is proportional to that between the corresponding dual variables: 

\begin{corollary}\label{cor:xx<=uulinear}
Consider the linearly constrained problem~\eqref{eq:linearprimalprob} under Assumption~\ref{asm:generalprimalprob}. Then, for any $u,v\in\mathbb{R}^{m+p}$,
\begin{align*}
\|\bar{x}(u)-\bar{x}(v)\|\le\frac{\sigma_{\max}(\tilde{A})}{\theta}\|u-v\|.
\end{align*}
\end{corollary}

\begin{proof}
See Appendix~\ref{ssec:proofcorxx<=uulinear}.
\end{proof}

Since the inequality constraints are linear in~\eqref{eq:linearprimalprob}, the bound provided in Corollary~\ref{cor:xx<=uulinear} is independent of $u$ and $v$. Moreover, it is tighter than that in Lemma~\ref{lem:xx<=cuu}, \emph{i.e.}, $\frac{\sigma_{\max}(\tilde{A})}{\theta}\le\min\{\gamma(u),\gamma(v)\}$. This can be seen from the facts that $\sup_{q\in G(u)}\|q\|\ge\frac{1}{\sqrt{m}}\|A'\|_F\ge\frac{1}{\sqrt{m}}\sigma_{\max}(A')$ and that $\sigma_{\max}^2(\tilde{A})\le\sigma_{\max}^2(A')+\sigma_{\max}^2(A)$. When there are no inequality constraints, \emph{i.e.}, $m=0$, the two bounds are equal. 

Due to the linearity of the constraints, the gradient of the dual function is globally Lipschitz continuous over the whole space with the Lipschitz constant $\frac{\sigma_{\max}^2(\tilde{A})}{\theta}$ \cite{Beck14}. Based on this, we provide global error bounds on primal optimality and feasibility:

\begin{proposition}\label{pro:primalfunctionvaluelinear}
Consider the linearly constrained problem~\eqref{eq:linearprimalprob} under Assumption~\ref{asm:generalprimalprob}. Then, for any $u\in D$ and any $u^\star\in D^\star$,
\begin{align}
-\|u^\star\|\sigma_{\max}(\tilde{A})\sqrt{\frac{2(d^\star-d(u))}{\theta}}&\le f(\bar{x}(u))-f^\star\le\|u\|\sigma_{\max}(\tilde{A})\sqrt{\frac{2(d^\star-d(u))}{\theta}},\label{eq:ff<=u2lambdamaxAAthetadd}\displaybreak[0]\\
\Delta(\bar{x}(u))&\le\sigma_{\max}(\tilde{A})\sqrt{\frac{2(d^\star-d(u))}{\theta}}.\label{eq:Delta<=2sigmathetadd}
\end{align}
\end{proposition}

\begin{proof}
See Appendix~\ref{ssec:proofproprimalfunctionvaluelinear}.
\end{proof}

Again, it can be shown that the upper bound in \eqref{eq:ff<=u2lambdamaxAAthetadd} is not specialized from and is tighter than that in \eqref{eq:ff<=usqrt2Lmp2sqrtdddd}.

\section{Primal convergence in dual first-order methods}\label{sec:CPI}

In this section, we use the connections between primal and dual errors that are built in Section~\ref{sec:PEB} 
to analyze the convergence properties of the approximate primal solution when some common first-order methods are employed to solve the dual problem. 

\subsection{Primal convergence in the projected dual gradient method}\label{ssec:dualgrad}
 
We first consider the projected dual gradient method. Let the dual iterates $(u_k)_{k=0}^\infty\subset D$ be generated by
\begin{align}
u_{k+1}=\mathcal{P}_D[u_k+\alpha\nabla d(u_k)],\quad\forall k\ge0\label{eq:gradientdescentnonlinear}
\end{align}
from an arbitrary initial point $u_0\in D$. To derive the convergence rates of $(u_k)_{k=0}^\infty$ and $(\bar{x}(u_k))_{k=0}^\infty$, we impose the following assumption:

\begin{assumption}\label{asm:inequalityconstraints} Problem~\eqref{eq:generalprimalprob} satisfies the following:
\begin{itemize}
\item[(a)] The constraint functions $g^{(i)}$ $\forall i\in\{1,2,\ldots,m\}$ are differentiable at every point in $X$.
\item[(b)] There exists $\tilde{u}\in\mathbb{R}^{m+p}$ such that $\tilde{u}^{(1:m)}\in\mathbb{R}^m_-$ and $\mathcal{L}(\cdot,\tilde{u})$ is strongly convex over $X$.
\end{itemize}
\end{assumption}

To satisfy Assumption~\ref{asm:inequalityconstraints}(b), it suffices that each $\nabla g^{(i)}$ satisfies a Lipschitz condition on $X$ with Lipschitz constant $L'_i\ge0$. To see this, note from the proof of \cite[Lemma~1.2.3]{Nesterov04} that for each $i\in\{1,2,\ldots,m\}$, $g^{(i)}(x_1)-g^{(i)}(x_2)-\nabla g^{(i)}(x_2)^T(x_1-x_2)\le\frac{L_i'}{2}\|x_1-x_2\|^2$ $\forall x_1,x_2\in X$. Hence, by letting $\tilde{u}\in\mathbb{R}^{m+p}$ be such that $\tilde{u}^{(1:m)}\in\mathbb{R}^m_-$ and $-\sum_{i=1}^m\tilde{u}^{(i)}L_i'<\theta$, we have
\begin{align*}
\mathcal{L}(x_1,\tilde{u})-\mathcal{L}(x_2,\tilde{u})-\tilde{\nabla}_x\mathcal{L}(x_2,\tilde{u})^T(x_1-x_2)\ge\frac{\theta+\sum_{i=1}^m\tilde{u}^{(i)}L_i'}{2}\|x_1-x_2\|^2,
\end{align*}
for each $x_1,x_2\in X$ and each subgradient $\tilde{\nabla}_x\mathcal{L}(x_2,\tilde{u})\in\partial_x\mathcal{L}(x_2,\tilde{u})$. Thus, $\mathcal{L}(\cdot,\tilde{u})$ is strongly convex over $X$. Using $\tilde{u}$ in Assumption~\ref{asm:inequalityconstraints}, we define the set
\begin{align*}
\tilde{D}&=\{u\in\mathbb{R}^{m+p}:u^{(1:m)}-\tilde{u}^{(1:m)}\in\mathbb{R}^m_+\}\supset D.
\end{align*} 
For any $u\in\tilde{D}$, $\mathcal{L}(\cdot,u)$ is strongly convex over $X$ and thus $\bar{x}(u)$ uniquely exists.

The next lemma is an important step toward establishing the convergence rates of $(u_k)_{k=0}^\infty$ and $(\bar{x}(u_k))_{k=0}^\infty$: 

\begin{lemma}\label{lem:dualgradientLipschitzlargerset}
Consider problem~\eqref{eq:generalprimalprob} under Assumptions~\ref{asm:generalprimalprob} and~\ref{asm:inequalityconstraints}. Then, for any $u\in D$ and $v\in\tilde{D}$,
\begin{align}
&\|\nabla d(u)-\nabla d(v)\|\le\frac{\sqrt{m+1}}{\theta}\Bigl(\sigma^2_{\max}(A)+\sum_{i=1}^mL_i^2\Bigr)^{1/2}\nonumber\\
&\quad\cdot\max\Bigl\{\sigma_{\max}(A),\max_{i\in\{1,\ldots,m\}}\|\nabla g^{(i)}(\bar{x}(u))\|, \max_{i\in\{1,\ldots,m\}}\|\nabla g^{(i)}(\bar{x}(v))\|\Bigr\}\|u-v\|.\label{eq:gdgd<=Luulargerset}
\end{align}
\end{lemma}

\begin{proof}
See Appendix~\ref{ssec:prooflemdualgradientLipschitzlargerset}.
\end{proof}

The Lipschitz-like property of $\nabla d$ on $\tilde{D}$ established in Lemma~\ref{lem:dualgradientLipschitzlargerset} will be used to derive further inequalities below. To present these, we need to introduce the following additional notation: For any convex and compact set $S\subset D$, let
\begin{align}
\Psi(S)\!=\!\begin{cases}\operatorname{conv}\Bigl(\Bigl\{u\!+\!\beta(\nabla d(u)\!-\!\nabla d(v)):u,v\in S,\; \beta \in [0, \frac{1}{\eta(S)}]\Bigr\}\Bigr),&\text{if $\eta(S)\!>\!0$},\\S&\text{otherwise},\end{cases}\label{eq:tD'}
\end{align}
where $\eta(S)\in[0,\infty)$ is defined by
\begin{align}
\eta(S)=\begin{cases}\sup\limits_{u,v\in S}\max\limits_{i\in\{1,\ldots,m\}}&\tfrac{-|\nabla^{(i)}d(u)-\nabla^{(i)}d(v)|}{\tilde{u}^{(i)}},\\
&\text{ if $\sup\limits_{u,v\in S}\max\limits_{i\in\{1,\ldots,m\}}\tfrac{-|\nabla^{(i)}d(u)-\nabla^{(i)}d(v)|}{\tilde{u}^{(i)}}>0$},\\ 
\frac{\sigma_{\max}^2(A)}{\theta},&\text{ otherwise}.\end{cases}\label{eq:etaD'}
\end{align}
This guarantees that $\Psi(S)$ is compact and $S\subseteq\Psi(S)\subset\tilde{D}$. The expression of $\eta(S)$ is admittedly complicated, but it allows us to include the pathological cases that $\nabla d(u)$ is constant over $S$ and that there are no inequality constraints (\emph{i.e.}, $m=0$). In particular, $\eta(S)=0$ means the absence of equality constraints (\emph{i.e.}, $p=0$) and the invariance of $\nabla d(u)$ on $S$. In this case, the above definition still guarantees that $u+\eta^{-1}(\nabla d(u)-\nabla d(v))\in\Psi(S)$ $\forall u,v\in S$ $\forall\eta>\eta(S)$.

Under Assumption~\ref{asm:inequalityconstraints}, the definitions of $G(u)$, $\gamma(u)$, and $L(S)$ in \eqref{eq:gammau}, \eqref{eq:Gu}, and \eqref{eq:LD'} can be extended to hold for any $u\in\tilde{D}$ and any compact set $S\subset\tilde{D}$. Also, Lemma~\ref{lem:xbounded} still holds when $D$ is replaced by $\tilde{D}$, which implies that $0<L(\Psi(S))<\infty$. Moreover, Lemma~\ref{lem:dualgradientLipschitzlargerset} implies that $\|\nabla d(u)-\nabla d(v)\|\le L(\Psi(S))\|u-v\|$ $\forall u\in S$ $\forall v\in\Psi(S)$. With these observations, consider the following lemma:

\begin{lemma}\label{lem:dualgradientLipschitzineq}
Consider problem~\eqref{eq:generalprimalprob} under Assumptions~\ref{asm:generalprimalprob} and~\ref{asm:inequalityconstraints}. Let $S\subset D$ be convex and compact. Also let $\eta(S)\in[0,\infty)$, $\Psi(S)\subset\tilde{D}$, and $L(\Psi(S))\in(0,\infty)$ be defined in \eqref{eq:etaD'}, \eqref{eq:tD'}, and \eqref{eq:LD'}, respectively. Then, for any $u\in S$ and $v\in \Psi(S)$,
\begin{align}
d(v)-d(u)-\nabla d(u)^T(v-u)&\ge-\frac{L(\Psi(S))}{2}\|u-v\|^2.\label{eq:ddgduu>=-M2uu}\displaybreak[0]\\
\intertext{Moreover, for any $u,v\in S$ and any $\eta>0$ such that $\eta\ge\eta(S)$,}
d(v)-d(u)-\nabla d(u)^T(v-u)&\le\Bigl(\frac{L(\Psi(S))}{2\eta^2}-\frac{1}{\eta}\Bigr)\|\nabla d(u)-\nabla d(v)\|^2,\displaybreak[0]\label{eq:ddgduu<=-M2eta1etagdgd}\\
(\nabla d(u)-\nabla d(v))^T(u-v)&\le2\Bigl(\frac{L(\Psi(S))}{2\eta^2}-\frac{1}{\eta}\Bigr)\|\nabla d(u)-\nabla d(v)\|^2.\label{eq:gdgduu<=-2M2eta1etagdgd}
\end{align}
\end{lemma}

\begin{proof}
See Appendix~\ref{ssec:prooflemdualgradientLipschitzineq}.
\end{proof}

\begin{remark}
Lemma~\ref{lem:dualgradientLipschitzineq} is critical in deriving the convergence rates of the projected dual gradient method~\eqref{eq:gradientdescentnonlinear}. Note that Theorem~2.1.5 in \cite{Nesterov04} gives similar inequalities as \eqref{eq:ddgduu>=-M2uu}, \eqref{eq:ddgduu<=-M2eta1etagdgd}, and \eqref{eq:gdgduu<=-2M2eta1etagdgd}. However, those inequalities require that $\nabla d$ is Lipschitz continuous over $\mathbb{R}^{m+p}$ and their proofs do not apply to our case where $\nabla d$ is locally Lipschitz continuous on $D$. Indeed, as is suggesed by Example~\ref{ex:linear} below, when problem~\eqref{eq:generalprimalprob} reduces to the linearly constrained problem~\eqref{eq:linearprimalprob}, Lemma~\ref{lem:dualgradientLipschitzineq} is specialized to Theorem~2.1.5 in \cite{Nesterov04}.
\end{remark}

Having established the inequalities in Lemmas~\ref{lem:dualgradientLipschitzlargerset} and~\ref{lem:dualgradientLipschitzineq}, we provide the dual and primal convergence rates for the projected dual gradient method \eqref{eq:gradientdescentnonlinear}:

\begin{theorem}\label{thm:dualprimalconvrate}
Consider problem~\eqref{eq:generalprimalprob} under Assumptions~\ref{asm:generalprimalprob} and~\ref{asm:inequalityconstraints}. Let $(u_k)_{k=0}^\infty\subset D$ be a sequence generated by the projected dual gradient method \eqref{eq:gradientdescentnonlinear}. Also, let $u^\star\in D^\star$ and $D_0=\{u\in D:\|u-u^\star\|\le\|u_0-u^\star\|\}\subset D$. Moreover, let $\Phi(D_0)\subset D$ be defined in \eqref{eq:D''}, $L(\Phi(D_0))\in(0,\infty)$ in \eqref{eq:LD'}, $\eta(D_0)\in[0,\infty)$ in \eqref{eq:etaD'}, $\Psi(D_0)\subset\tilde{D}$ in \eqref{eq:tD'}, and $L(\Psi(D_0))\in(0,\infty)$ in \eqref{eq:LD'}. If
\begin{align}
0<\alpha<\begin{cases}\frac{2}{L(\Psi(D_0))},&\text{ if }L(\Psi(D_0))>\eta(D_0),\\ 4\Bigl(\frac{1}{\eta(D_0)}-\frac{L(\Psi(D_0))}{2\eta(D_0)^2}\Bigr),&\text{ otherwise},\end{cases}\label{eq:stepsize}
\end{align}
then for any $k\ge0$,
\begin{align}
&d^\star-d(u_k)\le\frac{R_0}{1+ k R_0\delta\rho^{-1}},\label{eq:gpmdualconvrate}\displaybreak[0]\\
&\|\bar{x}(u_k)-x^\star\|\le\left(\frac{2R_0\theta^{-1}}{1+kR_0\delta\rho^{-1}}\right)^{1/2},\label{eq:gpmprimalconvdist}\displaybreak[0]\\
&f(\bar{x}(u_k))-f^\star\le(\|u^\star\|+\|u_0-u^\star\|)\left(\frac{2(m+p)L(\Phi(D_0))R_0}{1+kR_0\delta\rho^{-1}}\right)^{1/2}\nonumber\displaybreak[0]\\
&\qquad\qquad\qquad\quad\;+\frac{R_0}{1+k R_0\delta\rho^{-1}},\label{eq:gpmprimalconvratefunc}\displaybreak[0]\\
&f(\bar{x}(u_k))-f^\star\ge-\|u^\star\|\left(\frac{2L(\Phi(D_0))R_0}{1+ k R_0\delta\rho^{-1}}\right)^{1/2},\label{eq:gpmprimalconvratefunclb}\displaybreak[0]\\
&\Delta(\bar{x}(u))\le\left(\frac{2L(\Phi(D_0))R_0}{1+ k R_0\delta\rho^{-1}}\right)^{1/2},\label{eq:gpminfconvrate}
\end{align}
where $R_0=d^{\star}-d(u_0)\in(0,\infty)$, $\rho=(\sup_{u\in D_0}\|\nabla d(u)\|+\|u_0-u^\star\|/\alpha)^2\in(0,\infty)$, and $\delta=1/\alpha-L(\Psi(D_0))/2\in(0,\infty)$.
\end{theorem}

\begin{proof}
See Appendix~\ref{ssec:proofthmdualprimalconvrate}.
\end{proof}

Theorem~\ref{thm:dualprimalconvrate} says that under Assumptions~\ref{asm:generalprimalprob} and~\ref{asm:inequalityconstraints} as well as a proper step-size choice \eqref{eq:stepsize}, the dual function value at the dual iterates $(u_k)_{k=0}^\infty$ converges to $d^\star$ at a rate of $O(1/k)$.  
Note that this result extends earlier analysis of the projected gradient method for functions with globally Lipschitz continuous gradient (e.g., \cite{Levitin66}) to a class of functions with locally Lipschitz continuous gradient on closed and convex sets in the form of $D$. Moreover, this result implies that the primal iterates $(\bar{x}(u_k))_{k=0}^\infty$ converge at a rate no worse than $O(1/\sqrt{k})$ in primal optimality and feasibility. Furthermore, although \eqref{eq:stepsize} provides a sufficient condition for the range of step-sizes that guarantees these convergence rates, it does not explicitly tell how to select a proper step-size. In the examples below, we show explicit step-size rules satisfying \eqref{eq:stepsize} for some important problem classes.

\begin{example}\label{ex:Xcompact}
Suppose that $X$ is a compact set. Note that if $m\neq0$,
\begin{align}
\sup\limits_{u,v\in S}\max\limits_{i\in\{1,\ldots,m\}}&\tfrac{-|\nabla^{(i)}d(u)-\nabla^{(i)}d(v)|}{\tilde{u}^{(i)}}\le\sup\limits_{u,v\in D_0}\max\limits_{i\in\{1,\ldots,m\}}-\frac{|g^{(i)}(\bar{x}(u))-g^{(i)}(\bar{x}(v))|}{\tilde{u}^{(i)}}\nonumber\\
&\le\sup\limits_{u,v\in D_0}\max\limits_{i\in\{1,\ldots,m\}}-\frac{L_i}{\tilde{u}^{(i)}}\|\bar{x}(u)-\bar{x}(v)\|\label{eq:etaxbar}\\
&\le\max\limits_{i\in\{1,\ldots,m\}}-\frac{L_i}{\tilde{u}^{(i)}}\operatorname{diam}(X),\nonumber
\end{align}
where the first inequality is due to \eqref{eq:generalgradd} and the second comes from Assumption~\ref{asm:generalprimalprob}(b).
Hence, 
\begin{align*}
\eta(D_0)\le\hat{\eta}\triangleq\max\left\{\frac{\sigma_{\max}^2(A)}{\theta},\max\limits_{i\in\{1,\ldots,m\}}-\frac{L_i}{\tilde{u}^{(i)}}\operatorname{diam}(X)\right\}.
\end{align*}
Also note that
\begin{align*}
&\gamma(u)\le\hat{\gamma}\triangleq\frac{\sqrt{m+1}}{\theta}\max\{\sigma_{\max}(A),\sup_{x\in X}\max_{i\in\{1,2,\ldots,m\}}\|\nabla g^{(i)}(x)\|\},\quad\forall u\in\tilde{D},\\
\intertext{and thus}
&L(\Psi(D_0))\le\hat{L}\triangleq\hat{\gamma}\Bigl(\sigma_{\max}^2(A)+\sum_{i=1}^mL_i^2\Bigr)^{1/2}.
\end{align*}
Since $X$ is compact, we have $0<\hat{\eta}<\infty$ and $0<\hat{L}<\infty$. Then, as long as
\begin{align}
0<\alpha<\begin{cases}\frac{2}{\hat{L}},&\text{ if }\hat{L}>\hat{\eta},\\ 4(\frac{1}{\hat{\eta}}-\frac{\hat{L}}{2\hat{\eta}^2}),&\text{ otherwise},\end{cases}\label{eq:stepsizesmaller}
\end{align}
the step-size $\alpha$ satisfies \eqref{eq:stepsize}. Notice that unlike $\eta(D_0)$ and $L(\Psi(D_0))$, the constants $\hat{\eta}$ and $\hat{L}$ can be directly determined from the primal problem. 
\end{example}

\begin{example}\label{ex:gLipschitz}
Suppose, for each $i\in\{1,2,\ldots,m\}$, that $g^{(i)}$ is Lipschitz continuous on an open set containing $X$ with Lipschitz constant $L_i>0$. Then, $\|\nabla g^{(i)}(x)\|\le L_i$ $\forall x\in X$ $\forall i\in\{1,2,\ldots,m\}$, which implies that
\begin{align*}
\gamma(u)\le\hat{\gamma}\triangleq\frac{\sqrt{m+1}}{\theta}\max\{\sigma_{\max}(A),
\max_{i\in\{1,2,\ldots,m\}}L_i\},\quad\forall u\in D.
\end{align*}
Due to \eqref{eq:etaxbar} and Lemma~\ref{lem:xx<=cuu},
\begin{align*}
\sup\limits_{u,v\in S}\max\limits_{i\in\{1,\ldots,m\}}\tfrac{-|\nabla^{(i)}d(u)-\nabla^{(i)}d(v)|}{\tilde{u}^{(i)}}&\le\max\limits_{i\in\{1,\ldots,m\}}-\frac{L_i}{\tilde{u}^{(i)}}\cdot\!\!\sup_{u\in D_0}\gamma(u)\cdot\!\!\!\sup_{u,v\in D_0}\|u-v\|\\
&\le\max\limits_{i\in\{1,\ldots,m\}}-\frac{L_i}{\tilde{u}^{(i)}}\hat{\gamma}\operatorname{diam}(D_0).
\end{align*}
Therefore,
\begin{align*}
\eta(D_0)\le\hat{\eta}\triangleq\max\left\{\frac{\sigma_{\max}^2(A)}{\theta},\max\limits_{i\in\{1,\ldots,m\}}-\frac{L_i}{\tilde{u}^{(i)}}\hat{\gamma}\operatorname{diam}(D_0)\right\}.
\end{align*}
Also, $L(\Psi(D_0))\le\hat{L}\triangleq\hat{\gamma}\Bigl(\sigma_{\max}^2(A)+\sum_{i=1}^mL_i^2\Bigr)^{1/2}$. Then, any step-size $\alpha$ satisfying \eqref{eq:stepsizesmaller} meets \eqref{eq:stepsize}. Here, the constants $\hat{\eta}$ and $\hat{L}$ solely depend on the primal problem as well as an upper bound on the diameter of the set $D_0$.
\end{example}

\begin{example}\label{ex:linear}
When problem~\eqref{eq:generalprimalprob} reduces to the linearly constrained problem~\eqref{eq:linearprimalprob}, it becomes a special case of Example~\ref{ex:gLipschitz}. In this case, Assumption~\ref{asm:inequalityconstraints}(b) holds for every $\tilde{u}\in\mathbb{R}^{m+p}$ with $\tilde{u}^{(1:m)}\in\mathbb{R}_-^m$. By taking $\tilde{u}^{(1:m)}$ sufficiently small, we can make $\hat{\eta}$ in Example~\ref{ex:gLipschitz} equal to $\sigma_{\max}^2(A)/\theta$. Thus, $\eta(D_0)\le\sigma_{\max}^2(A)/\theta\le\sigma_{\max}^2(\tilde{A})/\theta$. This, along with the fact that $\nabla d$ is Lipschitz continuous with Lipschitz constant $\sigma_{\max}^2(\tilde{A})/\theta$, implies that $0<\alpha<\frac{2\theta}{\sigma_{\max}^2(\tilde{A})}$ satisfies \eqref{eq:stepsize}. This step-size condition coincides with the standard one used to guarantee the convergence of the gradient methods when the objective function has globally Lipschitz continuous gradient \cite{Nesterov04}. With such a step-size, \eqref{eq:gpmdualconvrate}--\eqref{eq:gpminfconvrate} hold with $L(\Phi(D_0))=L(\Psi(D_0))=\sigma_{\max}^2(\tilde{A})/\theta$. Also, from \eqref{eq:ff<=u2lambdamaxAAthetadd}, we have a tighter upper bound on the primal convergence rate
\begin{align}
f(\bar{x}(u_k))-f^\star\le\sigma_{\max}(\tilde{A})(\|u^\star\|+\|u_0-u^\star\|)\left(\frac{2R_0\theta^{-1}}{1+kR_0\delta\rho^{-1}}\right)^{1/2}.\label{eq:ff<=sigmauuu2Rtheta1kRdeltarho}
\end{align}
\end{example}

It is known that the projected gradient method is able to converge linearly when the objective function is strongly convex \cite{Nesterov04}. Theorem~\ref{thm:xx<=cuu} thus suggests that the primal iterates could achieve linear convergence if the dual function is strongly concave. Indeed, if the subgradients of $f$ satisfy a Lipschitz condition on $X$ with Lipschitz constant $M>0$, then the dual function for problem~\eqref{eq:linearprimalprob} with $\tilde{A}$ having full row rank and $X=\mathbb{R}^n$ is strongly concave with concavity parameter $-\theta\sigma_{\min}^2(\tilde{A})/M^2<0$ \cite{Urruty96}. Therefore, for any $\alpha\in(0,2M^2\theta/(\theta^2\sigma_{\min}^2(\tilde{A})+M^2\sigma_{\max}^2(\tilde{A})]$, we have
\begin{align*}
\|u_k-u^\star\|&\le q^k\|u_0-u^\star\|,\\
\|\bar{x}(u_k)-x^\star\|&\le q^k\frac{\sigma_{\max}(\tilde{A})}{\theta}\|u_0-u^\star\|,
\end{align*}
where $q=\Bigl(1-\frac{2\alpha\theta\sigma_{\min}^2(\tilde{A})\sigma_{\max}^2(\tilde{A})}{\theta^2\sigma_{\min}^2(\tilde{A})+M^2\sigma_{\max}^2(\tilde{A})}\Bigr)^{1/2}\in[0,1)$. Moreover, $q$ reaches its minimum\linebreak[4] $\Bigl(1-\frac{4M^2\theta^2\sigma_{\min}^2(\tilde{A})\sigma_{\max}^2(\tilde{A})}{(\theta^2\sigma_{\min}^2(\tilde{A})+M^2\sigma_{\max}^2(\tilde{A}))^2}\Bigr)^{1/2}$ when $\alpha=\frac{2M^2\theta}{\theta^2\sigma_{\min}^2(\tilde{A})+M^2\sigma_{\max}^2(\tilde{A})}$ \cite{Nesterov04}.

\subsection{Primal convergence in fast dual gradient methods}\label{ssec:dualfast}

In this subsection, we move on to fast dual gradient methods. To the best of the authors' knowledge, all the existing fast gradient methods require that the gradient of the objective function satisfies a Lipschitz condition on at least the feasible region in order to reach a convergence rate of $O(1/k^2)$. Hence, throughout this subsection, we let Assumptions~\ref{asm:generalprimalprob} and~\ref{asm:boundedsubgrad} hold, so that $\nabla d$ satisfies a Lipschitz condition on $D$ with Lipschitz constant $\tilde{L}$ defined in Corollary~\ref{cor:primalfunctionvalueboundedsubgrad}. We also assume that (an upper bound on) $\sup_{u\in D}\gamma(u)$ and thus $\tilde{L}$ are known.\footnote{In Examples~\ref{ex:Xcompact} and~\ref{ex:gLipschitz}, the upper bounds on $\sup_{u\in D}\gamma(u)$ and $\tilde{L}$ (\emph{i.e.}, $\hat{\gamma}$ and $\hat{L}$) can be easily computed from the primal problem.}

We consider the $1$-memory fast gradient method in \cite{Tseng08} for solving the dual problem~\eqref{eq:generaldualprob}. To start with, define the following: Let $h:\mathbb{R}^{m+p}\rightarrow\mathbb{R}$ be differentiable on an open set containing $D$ and let $Q(u,v)=h(u)-h(v)-\nabla h(v)^T(u-v)$ $\forall u\in\mathbb{R}^{m+p}$ $\forall v\in D$. Assume $h$ is strictly convex and satisfies $Q(u,v)\ge\|u-v\|^2/2$ $\forall u,v\in D$. Also, let $\ell_{-d}(u,v)=-d(v)-\nabla d(v)^T(u-v)$ $\forall u,v\in D$. For completeness, we provide the algorithm below:

\begin{algorithm}[Algorithm~1, \cite{Tseng08}]\label{alg:1memory}
\begin{algorithminit}{}
\item Let $\beta_0=1$ and choose $u_0,w_0\in D$.
\end{algorithminit}
\begin{algorithmoper}{At each time $k\ge0$:}
\item Choose a closed convex set $U_k\subseteq\mathbb{R}^{m+p}$ such that $U_k\cap D^\star\neq\emptyset$.
\item Let $v_k=(1-\beta_k)u_k+\beta_kw_k$,\\
\mbox{\quad\;\;\,}$w_{k+1}=\operatorname{arg\;min}_{u\in U_k\cap D}\ell_{-d}(u,v_k)+\beta_k\tilde{L}Q(u,w_k)$,\\
\mbox{\quad\;\;\,}$\hat{u}_{k+1}=(1-\beta_k)u_k+\beta_kw_{k+1}$.
\item Choose $u_{k+1}\in D$ be such that\\ 
$
\ell_{-d}(u_{k+1},v_k)+\frac{\tilde{L}}{2}\|u_{k+1}-v_k\|^2\le\ell_{-d}(\hat{u}_{k+1},v_k)+\frac{\tilde{L}}{2}\|\hat{u}_{k+1}-v_k\|^2.$
\item Choose $\beta_{k+1}\le2/(k+3)$.
\end{algorithmoper}
\end{algorithm}

In Algorithm~\ref{alg:1memory}, the variables $u_k$, $v_k$, and $w_k$ remain in $D$ at all times. One simple way to choose $U_k$ in Step~2 and $u_{k+1}$ in Step~4 is that $U_k\supseteq D$ and $u_{k+1}=\operatorname{arg\;min}_{u\in D}\ell_{-d}(u,v_k)+\tilde{L}Q(u,w_k)$. In this case, if $Q(u,v)=\|u-v\|^2/2$ , then the updates of $w_{k+1}$ and $u_{k+1}$ reduce to projected gradient steps $w_{k+1}=\mathcal{P}_D[w_k+\nabla d(v_k)/(\beta_k\tilde{L})]$ and $u_{k+1}=\mathcal{P}_D[v_k+\nabla d(v_k)/\tilde{L}]$. Other options for $U_k$ and $u_{k+1}$ can also be found in \cite{Tseng08}. Note that the above algorithm is indeed specialized from the more general Algorithm~1 in \cite{Tseng08}. This is for the purpose of deriving the primal and dual convergence rates in the following proposition:

\begin{proposition}\label{pro:fgdualprimalconvrate}
Consider problem~\eqref{eq:generalprimalprob} under Assumptions~\ref{asm:generalprimalprob} and~\ref{asm:boundedsubgrad}. Let\linebreak[4]$(u_k)_{k=0}^\infty\subset D$ be a sequence generated by Algorithm~\ref{alg:1memory} and let $u^\star\in D^\star$. Then, for any $k\ge1$,
\begin{align}
d^\star-d(u_k)&\le\frac{4\tilde{L}Q(u^\star,w_0)}{(k+1)^2},\label{eq:fgdualconvrate}\displaybreak[0]\\
\|\bar{x}(u_k)-x^\star\|&\le\frac{(8\tilde{L}Q(u^\star,w_0)\theta^{-1})^{1/2}}{k+1},\label{eq:fgprimalconvratedist}\displaybreak[0]\\
f(\bar{x}(u_k))-f^\star&\le\frac{\tilde{L}\|u_k\|_\infty\bigl(8(m+p)Q(u^\star,w_0)\bigr)^{1/2}}{k+1}+\frac{4\tilde{L}Q(u^\star,w_0)}{(k+1)^2},\label{eq:fgprimalconvratefunc}\displaybreak[0]\\
f(\bar{x}(u_k))-f^\star&\ge-\frac{\tilde{L}\|u^\star\|(8Q(u^\star,w_0))^{1/2}}{k+1},\label{eq:fgprimalconvratefunclb}\displaybreak[0]\\
\Delta(\bar{x}(u_k))&\le\frac{\tilde{L}(8Q(u^\star,w_0))^{1/2}}{k+1},\label{eq:fginfconvrate}
\end{align}
where $\tilde{L}$ is defined in Corollary~\ref{cor:primalfunctionvalueboundedsubgrad}.
\end{proposition}

\begin{proof}
See Appendix~\ref{ssec:proofprofgdualprimalconvrate}.
\end{proof}

Proposition~\ref{pro:fgdualprimalconvrate} says that Algorithm~\ref{alg:1memory} yields $O(1/k^2)$ convergence rate of $(u_k)_{k=1}^\infty$ in dual optimality. In addition, the distance between $\bar{x}(u_k)$ and $x^\star$ as well as the primal infeasibility of $\bar{x}(u_k)$ vanishes at a rate no worse than $O(1/k)$. As Algorithm~\ref{alg:1memory} does not guarantee that $(u_k)_{k=1}^\infty$ is bounded, it says nothing about the convergence rate of $f(\bar{x}(u_k))$. Nevertheless, if problem~\eqref{eq:generalprimalprob} has only inequality constraints, then the dual optimal set is bounded \cite{Nedic09b} and so is $(u_k)_{k=1}^\infty$. This leads to the following proposition, which states that in the absence of equality constraints, $f(\bar{x}(u_k))$ converges to $f^\star$ at a rate $O(1/k)$ after some finite time:

\begin{proposition}\label{pro:fgdualprimalconvrateineq}
Consider problem~\eqref{eq:generalprimalprob} under Assumptions~\ref{asm:generalprimalprob} and~\ref{asm:boundedsubgrad}. Suppose $p=0$. Let $(u_k)_{k=0}^\infty\subset D$ be a sequence generated by Algorithm~\ref{alg:1memory}. Also, let $\tilde{x}\in\mathbb{R}^n$ satisfy Assumption~\ref{asm:generalprimalprob}(c), $u^\star\in D^\star$, and $\bar{u}\in D\backslash D^\star$. Then,
\begin{align}
f(\bar{x}(u_k))-f^\star\le&\frac{\tilde{L}(d(\bar{u})-f(\tilde{x}))\bigl(8(m+p)Q(u^\star,w_0)\bigr)^{1/2}}{(k+1)(\max_{i\in\{1,2,\ldots,m\}}g^{(i)}(\tilde{x}))}+\frac{4\tilde{L}Q(u^\star,w_0)}{(k+1)^2},\nonumber\displaybreak[0]\\
&\qquad\qquad\qquad\qquad\qquad\qquad\forall k>\Bigl(\frac{4\tilde{L}Q(u^\star,w_0)}{d^\star-d(\bar{u})}\Bigr)^{1/2}.\label{eq:fgprimalconvratefuncineq}
\end{align}
\end{proposition}

\begin{proof}
See Appendix~\ref{ssec:proofprofgdualprimalconvrateineq}.
\end{proof}

\begin{remark}
In addition to Algorithm~\ref{alg:1memory}, \emph{i.e.}, the $1$-memory fast gradient method in \cite{Tseng08}, the dual problem~\eqref{eq:generaldualprob} can also be solved by the $\infty$-memory fast gradient method in \cite{Tseng08}, which would produce similar primal and dual convergence rates as those in Propositions~\ref{pro:fgdualprimalconvrate} and~\ref{pro:fgdualprimalconvrateineq}. Due to space limitation and since such analyses are very similar to that in Propositions~\ref{pro:fgdualprimalconvrate} and~\ref{pro:fgdualprimalconvrateineq}, we omit this algorithm in the paper.
\end{remark}

Recall that for the linearly constrained problem~\eqref{eq:linearprimalprob}, the dual function $d$ is differentiable and has globally Lipschitz continuous gradient with Lipschitz constant $\sigma_{\max}^2(\tilde{A})/\theta$. In this case, the following algorithm, which has a simpler form than Algorithm~\ref{alg:1memory}, can be adopted to solve the dual problem:

\begin{algorithm}[Algorithm~2, \cite{Tseng08}]\label{alg:1memorywholespace}
\begin{algorithminit}{}
\item Let $\beta_0=\beta_{-1}=1$ and choose $u_0=u_{-1}\in D$.
\end{algorithminit}
\begin{algorithmoper}{At each time $k\ge0$:}
\item Choose a closed convex set $U_k\subseteq\mathbb{R}^{m+p}$ such that $U_k\cap D^\star\neq\emptyset$.
\item Let $v_k=u_k+\beta_k(1/\beta_{k-1}-1)(u_k-u_{k-1})$ and \\
\mbox{\quad\;\;\,}$u_{k+1}=\operatorname{arg\;min}_{u\in U_k\cap D}\ell_{-d}(u,v_k)+\frac{\sigma_{\max}^2(\tilde{A})}{2\theta}\|u-v_k\|^2$.
\item Choose $\beta_{k+1}\le2/(k+3)$.
\end{algorithmoper}
\end{algorithm}

If we pick $U_k\supseteq D$, then the update of $u_{k+1}$ in Step~3 is a projected gradient step $u_{k+1}=\mathcal{P}_D[v_k+\nabla d(v_k)\theta/\sigma_{\max}^2(\tilde{A})]$. If we also choose $\beta_{k+1}=\left(\sqrt{\beta_k^4+4\beta_k^2}-\beta_k^2\right)/2$, then Algorithm~\ref{alg:1memorywholespace} becomes the fast iterative shrinkage-thresholding algorithm (FISTA) in \cite{Beck09} applied to solve the dual problem. The primal and dual convergence rates of Algorithm~\ref{alg:1memorywholespace}, which have the same order as Algorithm~\ref{alg:1memory} but have a more explicit form, are given below:

\begin{proposition}\label{pro:fgdualprimalconvratewholdspace}
Consider the linearly constrained problem~\eqref{eq:linearprimalprob} under Assumption~\ref{asm:generalprimalprob}. Let $(u_k)_{k=0}^\infty\subset D$ be a sequence generated by Algorithm~\ref{alg:1memorywholespace} and $u^\star\in D^\star$. Then, for any $k\ge1$,
\begin{align}
d^\star-d(u_k)&\le\frac{2\sigma_{\max}^2(\tilde{A})\|u_0-u^\star\|^2}{\theta(k+1)^2},\label{eq:fgdualconvratewholespace}\displaybreak[0]\\
\|\bar{x}(u_k)-x^\star\|&\le\frac{2\sigma_{\max}(\tilde{A})\|u_0-u^\star\|}{\theta(k+1)},\label{eq:fgprimalconvratedistwholespace}\displaybreak[0]\\
-\frac{2\|u^\star\|\sigma_{\max}^2(\tilde{A})\|u_0-u^\star\|}{\theta(k+1)}&\le f(\bar{x}(u_k))-f^\star\le\frac{2\|u_k\|\sigma_{\max}^2(\tilde{A})\|u_0-u^\star\|}{\theta(k+1)},\label{eq:fgprimalconvratefuncwholespace}\displaybreak[0]\\
\Delta(\bar{x}(u_k))&\le\frac{2\sigma_{\max}^2(\tilde{A})\|u_0-u^\star\|}{\theta(k+1)}.\label{eq:fginfconvratewholespace}
\end{align}
Moreover, if $p=0$, then
\begin{align}
f(\bar{x}(u_k))-f^\star\le&\frac{2\sigma_{\max}^2(\tilde{A})(d(\bar{u})-f(\tilde{x}))\|u_0-u^\star\|}{\theta(k+1)\max_{i\in\{1,2,\ldots,m\}}g^{(i)}(\tilde{x})},\nonumber\displaybreak[0]\\
&\qquad\qquad\forall k>\sigma_{\max}(\tilde{A})\|u_0-u^\star\|\left(\frac{2\theta^{-1}}{d^\star-d(\bar{u})}\right)^{1/2},\label{eq:fgprimalconvratefuncineqwholespace}
\end{align}
where $\tilde{x}$ and $\bar{u}$ are defined as in Proposition~\ref{pro:fgdualprimalconvrateineq}.
\end{proposition}

\begin{proof}
See Appendix~\ref{ssec:proofprofgdualprimalconvratewholdspace}.
\end{proof}

The $O(1/k)$ primal convergence rates in~\eqref{eq:fgprimalconvratedistwholespace} and~\eqref{eq:fginfconvratewholespace} for linearly constrained problem~\eqref{eq:linearprimalprob} are also provided in \cite{Giselsson13,Beck14}. Moreover, as is shown in \cite{Beck14}, $\sigma_{\max}(\tilde{A})$ can be replaced by $\|\tilde{A}\|_{2,\infty}\triangleq\{\|\tilde{A}x\|_\infty:\|x\|=1\}$.

Compared with the projected dual gradient method \eqref{eq:gradientdescentnonlinear}, the fast dual gradient methods considered above are capable of increasing the primal convergence rates from $O(1/\sqrt{k})$ to $O(1/k)$. However, the problems that these methods can handle must satisfy Assumption~\ref{asm:boundedsubgrad}, which is not necessary for projected dual gradient method. On the other hand, in order to guarantee the sublinear dual and primal convergence rates, the projected dual gradient method has to satisfy Assumption~\ref{asm:inequalityconstraints} while the fast dual gradient methods do not. Moreover, the fast dual gradient methods are more complicated to implement---in addition to solving $\min_{x\in X}\mathcal{L}(x,u_k)$ for constructing the approximate primal solution $\bar{x}(u_k)$ that is also needed in the projected dual gradient method, they have to solve $\min_{x\in X}\mathcal{L}(x,v_k)$ in order to compute $\nabla d(v_k)$. 

\section{Numerical example}\label{sec:numeexam}

In this section, we compare the dual and primal convergence performance of the dual first-order methods in Section~\ref{sec:CPI} and the double smoothing method \cite{Devolder12} via a numerical example. 

We consider the following model predictive control (MPC) problem, which has a very similar form as the one formulated in \cite{Giselsson13}:
\begin{align*}
\begin{array}{ll}\underset{x\in \mathbb{R}^{n}}{\mbox{minimize}} & f(x)=\frac{1}{2}x^THx+t^Tx+\gamma\|Px-s\|_1\\ \operatorname{subject\,to} & A_1x+b_1\le0,\\ &A_2x+b_2=0,\\ & x\in \{y\in\mathbb{R}^n:|y^{(i)}|\le r_i\;\forall i\in\{1,2,\ldots,n\}\},\end{array}
\end{align*}
where $H\in\mathbb{R}^{n\times n}$ is positive definite, $t\in\mathbb{R}^n$, $\gamma>0$, $P\in\mathbb{R}^{q\times n}$, $s\in\mathbb{R}^q$, $A_1\in\mathbb{R}^{m\times n}$, $b_1\in\mathbb{R}^m$, $A_2\in\mathbb{R}^{p\times n}$, $b_1\in\mathbb{R}^p$, $r_i>0$, all of which are randomly generated with $n=10$, $q=5$, $m=3$, and $p=2$. Note that such a linearly constrained problem belongs to the intersection of the problem classes that the projected dual gradient method \eqref{eq:gradientdescentnonlinear}, the fast dual gradient methods in Algorithms~\ref{alg:1memory} and~\ref{alg:1memorywholespace}, and the double smoothing method in \cite{Devolder12} can handle. Also, since Algorithm~\ref{alg:1memory} has similar convergence rates as Algorithm~\ref{alg:1memorywholespace} for this problem, we omit Algorithm~\ref{alg:1memory} to be able to visualize the results better.

For the projected dual gradient method \eqref{eq:gradientdescentnonlinear}, we choose the step-size\linebreak[4] $\alpha=2\lambda_{\min}(H)/\lambda_{\max}(A_1^TA_1+A_2^TA_2)\times99\%$, which satisfies the step-size condition in Example~\ref{ex:linear}. For the fast gradient method in Algorithm~\ref{alg:1memorywholespace}, we choose the parameters $U_k$ and $\beta_k$ to be such that this method reduces to FISTA \cite{Beck09}. For the double smoothing method \cite{Devolder12}, since the linear constraint of the problem class that this method can handle is in the form of $\mathcal{A}x\in T$ with $\mathcal{A}$ being a linear operator and $T$ being a compact set, we put $\mathcal{A}=A_2$, $T=\{b_2\}$, and view $\{y\in\mathbb{R}^n:A_1y+b_1\le0,\;|y^{(i)}|\le r_i\;\forall i=1,2,\ldots,n\}$ as its set constraint. Moreover, since one smoothing parameter in the double smoothing method relies on an upper bound on some dual optimum, we adopt its practical implementation version in \cite{Devolder12}, which starts with an initial guess of this upper bound and repeatedly applying the method to a sequence of doubly smoothed dual problems with increasing guess on the upper bound until a correct guess is achieved. We choose the desired accuracy $\epsilon$ of the method to be $0.05$.

In addition to the approximate primal solution $\bar{x}(u_k)$ studied in this paper, we also consider in the simulation the average $\tilde{x}_k\triangleq\sum_{\ell=0}^k\bar{x}(u_\ell)/k$ of the primal iterates as in \cite{Nedic09b} for the projected dual gradient method and FISTA, and the running average $\hat{x}_k\triangleq(\sum_{\ell=0}^k\beta_\ell^{-1}\bar{x}(u_\ell))/(\sum_{\ell=0}^k\beta_\ell^{-1})$ with the weights being $1/\beta_\ell$ as in \cite{Patrinos13} for FISTA. For the double smoothing method, $u_k$ is a sequence generated by a fast gradient method in \cite[Sec. 2.2.1]{Nesterov04} applied to the smoothed dual; the approximate primal solution $\bar{x}(u_k)$ for this specific problem is the unique minimizer of the Lagrangian of the original problem.

\begin{figure}[tb]
\centering\includegraphics[width=4in]{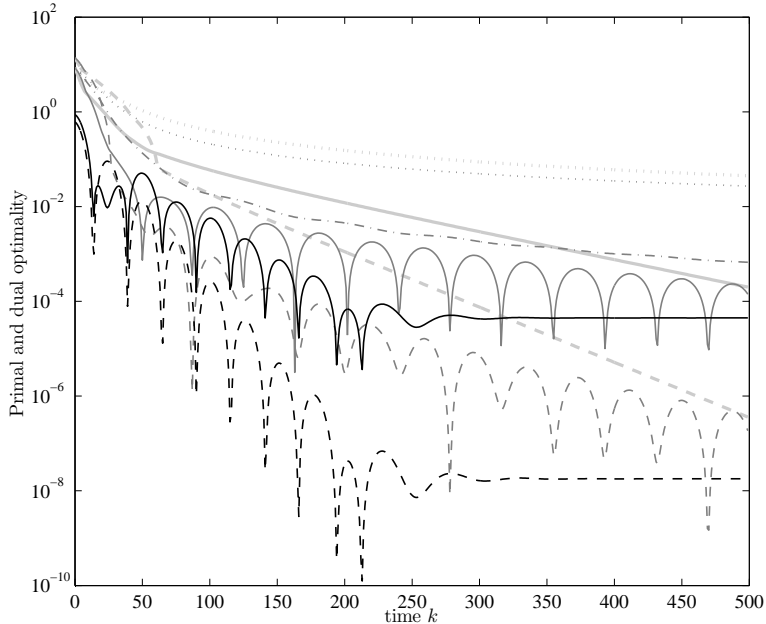}
\caption{Primal optimality of approximate primal solutions and dual optimality of dual iterates (The light grey, grey, and black dashed curves represent the dual optimality $d^\star-d(u_k)$ of $u_k$ in the projected dual gradient method, FISTA, and the double smoothing method, respectively. The light grey, grey, and black solid curves represent the primal optimality $|f(\bar{x}(u_k))-f^\star|$ of $\bar{x}(u_k)$ in the projected dual gradient method, FISTA, and the double smoothing method, respectively. The light grey and grey dotted curves represent the primal optimality $|f(\tilde{x}_k)-f^\star|$ of $\tilde{x}_k$ in the projected dual gradient method and FISTA, respectively. The grey dash-dotted curve represents the primal optimality $|f(\hat{x}_k)-f^\star|$ of $\hat{x}_k$ in FISTA.).}
\label{fig:primaldualopt}
\end{figure}

Figure~\ref{fig:primaldualopt} compares the convergence of the dual iterates generated by the three methods mentioned above, alongside with various choices for the approximate primal solution. Generally speaking, the dual convergence rate is faster than the primal in all of the three methods. Also, the primal iterates have faster convergence than their averages in the projected dual gradient method and FISTA. The projected dual gradient method converges slower than the other two methods in dual and primal optimality. FISTA and the double smoothing method have comparable performance, but the double smoothing method has the drawback that it only guarantees a prespecified target accuracy and does not ensure asymptotic convergence. 

\section{Conclusions}\label{sec:concl}

This paper studied primal convergence properties of dual first-order methods for solving optimization problems with a strongly convex objective function and general convex constraints including nonlinear inequality, linear equality, and set constraints. The unique minimizer of the Lagrangian at the current dual iterate, which is needed for evaluating the dual gradient, was considered as an approximate primal solution. The errors of this approximate primal solution, both in optimality and feasibility, were related to the dual errors. Sublinear dual and primal convergence rates for the projected dual gradient method and a few fast dual gradient methods were established. 

It is notable that this is the first work that ever provides primal convergence rates of dual first-order methods for \emph{nonlinearly} constrained convex optimization with \emph{locally} Lipschitz dual gradient. This work may also bring insights to future research on the convergence performance of other approximate primal solutions such as running averages of the primal iterates in dual first-order methods.

\section{Appendix}\label{sec:app}

\subsection{Proof of Lemma~\ref{lem:xbounded}}\label{ssec:prooflemxbounded}

Let $S\subset D$ be compact, $\bar{c}=\max_{u\in S}d(u)$, and $\underline{c}=\min_{u\in S}d(u)$, which are bounded due to the continuity of $d$. Also let $\partial_x\mathcal{L}$ denote the subdifferential of $\mathcal{L}$ with respect to the first argument. Fix $u_0\in S$. Since the function $\mathcal{L}(\cdot,u_0)$ is strongly convex over $X$, the sublevel set $\mathcal{C}\triangleq\{x\in X:\;\mathcal{L}(x,u_0)\le\bar{c}\}$ is nonempty and compact. Thus, $s\triangleq\max_{x\in\mathcal{C}}\sum_{i=1}^m(g^{(i)}(x))^2+\|Ax+b\|^2\in[0,\infty)$. Also let $r=\max_{u\in S}\|u-u_0\|\in[0,\infty)$. Then, for any $u\in S$ and any $x\in\mathcal{C}$,
\begin{align}
\mathcal{L}(x,u_0)-r\sqrt{s}\le\mathcal{L}(x,u)\le\mathcal{L}(x,u_0)+r\sqrt{s}.\label{eq:Lrsqrts<=L<=Lrsqrts}
\end{align}
To prove the boundedness of $\{\bar{x}(u):u\in S\}$, assume to the contrary that it is unbounded. Thus, given $x_0\in\mathcal{C}$, there exists $u'\in S$ such that $\|\bar{x}(u')-x_0\|^2>\frac{2(\mathcal{L}(x_0,u_0)+r\sqrt{s}-\underline{c})}{\theta}\ge0$. Since $\mathcal{L}(\cdot,u')$ is strongly convex over $X$ with convexity parameter $\theta$ and since there exists a subgradient $\tilde{\nabla}_x\mathcal{L}(\bar{x}(u'),u')\in\partial_x\mathcal{L}(\bar{x}(u'),u')$ satisfying $\tilde{\nabla}_x(\mathcal{L}(\bar{x}(u'),u'))^T(x-\bar{x}(u'))\ge0$ $\forall x\in X$, we have $\mathcal{L}(x_0,u')-\underline{c}\ge\mathcal{L}(x_0,u')-\mathcal{L}(\bar{x}(u'),u')\ge\frac{\theta}{2}\|\bar{x}(u')-x_0\|^2$. This gives $\mathcal{L}(x_0,u')>\mathcal{L}(x_0,u_0)+r\sqrt{s}$, which contradicts \eqref{eq:Lrsqrts<=L<=Lrsqrts}. Therefore, $\{\bar{x}(u):u\in S\}$ is bounded.

\subsection{Proof of Lemma~\ref{lem:xx<=cuu}}\label{ssec:prooflemxx<=cuu}

Let $u,v\in D$. From \cite[Theorem~3, Sec. 7.1.2]{Polyak87}, there exist subgradients $\tilde{\nabla}_x\mathcal{L}(\bar{x}(u),u)\in\partial_x\mathcal{L}(\bar{x}(u),u)$ and $\tilde{\nabla}_x\mathcal{L}(\bar{x}(v),v)\in\partial_x\mathcal{L}(\bar{x}(v),v)$ such that
\begin{align}
\tilde{\nabla}_x\mathcal{L}(\bar{x}(u),u)^T(\bar{x}(v)-\bar{x}(u))\ge0,\label{eq:L1x2x1>=0}\\
\tilde{\nabla}_x\mathcal{L}(\bar{x}(v),v)^T(\bar{x}(u)-\bar{x}(v))\ge0.\label{eq:L2x1x2>=0}
\end{align}
Due to \cite[Prop. 4.2.4]{Bertsekas03}, $\partial_x\mathcal{L}(x,u)=\partial f(x)+\sum_{i=1}^mu^{(i)}\partial g^{(i)}(x)+A^Tu^{(m+1:m+p)}$ $\forall x\in\mathbb{R}^n$ $\forall u\in D$. Thus, there exist subgradients $\tilde{\nabla}f(\bar{x}(u))\in\partial f(\bar{x}(u))$, $\tilde{\nabla}f(\bar{x}(v))\in\partial f(\bar{x}(v))$, and $\tilde{\nabla}g^{(i)}(\bar{x}(u))\in\partial g^{(i)}(\bar{x}(u))$, $\tilde{\nabla}g^{(i)}(\bar{x}(v))\in\partial g^{(i)}(\bar{x}(v))$, $\forall i\in\{1,2,\ldots,m\}$ such that
\begin{align*}
\tilde{\nabla}_x\mathcal{L}(x,u)=\tilde{\nabla}f(\bar{x}(u))+\sum_{i=1}^mu^{(i)}\tilde{\nabla}g^{(i)}(\bar{x}(u))+A^Tu^{(m+1:m+p)},\\
\tilde{\nabla}_x\mathcal{L}(x,v)=\tilde{\nabla}f(\bar{x}(v))+\sum_{i=1}^mv^{(i)}\tilde{\nabla}g^{(i)}(\bar{x}(v))+A^Tv^{(m+1:m+p)}.
\end{align*}
Adding \eqref{eq:L1x2x1>=0} and \eqref{eq:L2x1x2>=0} and substituting the above into the resulting inequality, we obtain
\begin{align}
\Bigl(\tilde{\nabla}f(\bar{x}(u))-\tilde{\nabla}f(\bar{x}(v))\Bigr)^T&(\bar{x}(u)-\bar{x}(v))\le\Bigl(u^{(m+1:m+p)}-v^{(m+1:m+p)}\Bigr)^TA(\bar{x}(v)-\bar{x}(u))\nonumber\\&+\Bigl(\sum_{i=1}^mu^{(i)}\tilde{\nabla}g^{(i)}(\bar{x}(u))-v^{(i)}\tilde{\nabla}g^{(i)}(\bar{x}(v))\Bigr)^T(\bar{x}(v)-\bar{x}(u)).\label{eq:tnftnfxx<=sumugugxx}
\end{align}
Note that
\begin{align}
&\Bigl(u^{(m+1:m+p)}-v^{(m+1:m+p)}\Bigr)^TA(\bar{x}(v)-\bar{x}(u))\nonumber\\
&\quad\le\|A^T(u^{(m+1:m+p)}-v^{(m+1:m+p)})\|\cdot\|\bar{x}(u)-\bar{x}(v)\|\nonumber\\
&\quad\le\sigma_{\max}(A)\|\bar{x}(u)-\bar{x}(v)\|\cdot\|u^{(m+1:m+p)}-v^{(m+1:m+p)}\|.\label{eq:uuAxx<=lambdamaxxxuu}
\end{align}
In addition,
\begin{align}
&\Bigl(\sum_{i=1}^mu^{(i)}\tilde{\nabla}g^{(i)}(\bar{x}(u))-v^{(i)}\tilde{\nabla}g^{(i)}(\bar{x}(v))\Bigr)^T(\bar{x}(v)-\bar{x}(u))\displaybreak[0]\nonumber\\
=&\sum_{i=1}^mu^{(i)}\Bigl(\tilde{\nabla}g^{(i)}(\bar{x}(u))-\tilde{\nabla}g^{(i)}(\bar{x}(v))\Bigr)^T(\bar{x}(v)-\bar{x}(u))\displaybreak[0]\nonumber\\
&+\Bigl(u^{(i)}-v^{(i)}\Bigr)\tilde{\nabla}g^{(i)}(\bar{x}(v))^T(\bar{x}(v)-\bar{x}(u))\displaybreak[0]\nonumber\\
\le&\sum_{i=1}^m\|\tilde{\nabla}g^{(i)}(\bar{x}(v))\|\cdot\|\bar{x}(u)-\bar{x}(v)\|\cdot|u^{(i)}-v^{(i)}|,\label{eq:sumuggxx<=sumtngxxuu}
\end{align}
where the inequality holds because for each $i\in\{1,2,\ldots,m\}$, $g^{(i)}$ is convex and $u^{(i)}\ge0$. 
From \eqref{eq:tnftnfxx<=sumugugxx}, \eqref{eq:sumuggxx<=sumtngxxuu}, \eqref{eq:uuAxx<=lambdamaxxxuu}, and the strong convexity of $f$ over $X$, we have
\begin{align*}
\theta\|\bar{x}(u)-\bar{x}(v)\|^2\le&\Bigl(\sum_{i=1}^m\|\tilde{\nabla}g^{(i)}(\bar{x}(v))\|\cdot\|\bar{x}(u)-\bar{x}(v)\|\cdot|u^{(i)}-v^{(i)}|\Bigr)\\
&+\sigma_{\max}(A)\|\bar{x}(u)-\bar{x}(v)\|\cdot\|u^{(m+1:m+p)}-v^{(m+1:m+p)}\|,
\end{align*}
which yields
\begin{align}
&\|\bar{x}(u)-\bar{x}(v)\|\le\frac{1}{\theta}\Bigl(\Bigl(\sum_{i=1}^m\sup_{q\in\partial g^{(i)}(\bar{x}(v))}\!\!\!\!\!\!\!\!\!\!\|q\|\!\cdot\!|u^{(i)}\!-\!v^{(i)}|\Bigr)\!+\!\sigma_{\max}(A)\|u^{(m+1:m+p)}\!-\!v^{(m+1:m+p)}\|\Bigr)\label{eq:xx<=uutight}\\
&\le\frac{1}{\theta}\max\{\sigma_{\max}(A),\sup_{q\in G(v)}\|q\|\}\Bigl(\Bigl(\sum_{i=1}^m|u^{(i)}-v^{(i)}|\Bigr)+\|u^{(m+1:m+p)}-v^{(m+1:m+p)}\|\Bigr)\displaybreak[0]\nonumber\\
&\le\frac{1}{\theta}\max\{\sigma_{\max}(A),\sup_{q\in G(v)}\|q\|\}\displaybreak[0]\nonumber\\
&\quad\cdot\Bigl[(m+1)\Bigl(\Bigl(\sum_{i=1}^m|u^{(i)}-v^{(i)}|^2\Bigr)+\|u^{(m+1:m+p)}-v^{(m+1:m+p)}\|^2\Bigr)\Bigr]^{1/2}\displaybreak[0]\nonumber\\
&=\gamma(v)\|u-v\|.\nonumber
\end{align}
Since $u$ and $v$ are interchangeable, \eqref{eq:xx<=uuloose} is satisfied.

\subsection{Proof of Theorem~\ref{thm:xx<=cuu}}\label{ssec:proofthmxx<=cuu}

Let $u\in D$ and $u^\star\in D^\star$. Then, by letting $v$ in Lemma~\ref{lem:xx<=cuu} be $u^\star$, we get \eqref{eq:xx<=gammauu}. To derive \eqref{eq:xx<=sqrtddtheta}, note that $\mathcal{L}(\cdot,u)$ is strongly convex on $X$ with convexity parameter $\theta$. Also note that there exists a subgradient $\tilde{\nabla}_x\mathcal{L}(\bar{x}(u),u)\in\partial_x\mathcal{L}(\bar{x}(u),u)$ such that $\tilde{\nabla}_x\mathcal{L}(\bar{x}(u),u)^T(x-\bar{x}(u))\ge0$ $\forall x\in X$. As a result, $\frac{\theta}{2}\|\bar{x}(u)-x^\star\|^2\le\mathcal{L}(x^\star,u)-\mathcal{L}(\bar{x}(u),u)-\tilde{\nabla}_x\mathcal{L}(\bar{x}(u),u)^T(x^\star-\bar{x}(u))\le\mathcal{L}(x^\star,u)-\mathcal{L}(\bar{x}(u),u)$.
From the Lagrangian saddle point theorem \cite[Prop. 6.2.4]{Bertsekas03}, $\mathcal{L}(x^\star,u)\le\mathcal{L}(x^\star,u^\star)=d^\star$. Also, $\mathcal{L}(\bar{x}(u),u)=d(u)$. Therefore, \eqref{eq:xx<=sqrtddtheta} holds.

\subsection{Proof of Proposition~\ref{pro:gradientdualLipschitz}}\label{ssec:proofprogradientdualLipschitz}

Let $S\subset D$ be compact. Due to Lemma~\ref{lem:xx<=cuu}, for any $u,v\in S$,
\begin{align*}
\|\nabla d(u)-\nabla d(v)\|^2&=\|A(\bar{x}(u)-\bar{x}(v))\|^2+\sum_{i=1}^m\|g^{(i)}(\bar{x}(u))-g^{(i)}(\bar{x}(v))\|^2\displaybreak[0]\\
&\le\sigma_{\max}^2(A)\|\bar{x}(u)-\bar{x}(v)\|^2+\sum_{i=1}^mL_i^2\|\bar{x}(u)-\bar{x}(v)\|^2\displaybreak[0]\\
&\le\sup_{w\in S}\gamma^2(w)\Bigl(\sigma_{\max}^2(A)+\sum_{i=1}^mL_i^2\Bigr)\|u-v\|^2.
\end{align*}
Therefore, \eqref{eq:gdgd<=Luu} holds. In addition, if $S$ is convex, the proof of \cite[Lemma~1.2.3]{Nesterov04} can be applied to obtain \eqref{eq:ddgduu>=-L2uu}.

\subsection{Proof of Theorem~\ref{thm:primalfunctionvalue}}\label{ssec:proofthmprimalfunctionvalue}

Let $S\subset D$ be compact, $u\in S$, and $u^\star\in D^\star$. From \eqref{eq:ff<=-gdu}, we know that to find an upper bound on $f(\bar{x}(u))-f^\star$ in terms of $d^\star-d(u)$, it is sufficient to do so to $-\nabla d(u)^Tu$.

Let ${\mathcal A}(u)=\{i\in\{1,2,\ldots,m\}:u^{(i)}+\tfrac{1}{L(\Phi(S))}\nabla^{(i)}d(u)<0\}$ and ${\mathcal I}(u)=\{1,2,\ldots,m+p\}\backslash {\mathcal A}(u)$. These sets identify the components of the dual vector $u$ for which, after a gradient step, the projection onto $D$ will be active and inactive, respectively. 
Since $-\nabla d(u)^Tu=\Bigl(\sum_{i\in\mathcal{A}(u)}-\nabla^{(i)}d(u)u^{(i)}\Bigr)+\Bigl(\sum_{i\in\mathcal{I}(u)}-\nabla^{(i)}d(u)u^{(i)}\Bigr)$, below we derive upper bounds on $\sum_{i\in\mathcal{A}(u)}-\nabla^{(i)}d(u)u^{(i)}$ and $\sum_{i\in\mathcal{I}(u)}-\nabla^{(i)}d(u)u^{(i)}$. 

To this end, let $D^{(i)}=[0,\infty)$ for $i=1,2,\ldots,m$ and $D^{(i)}=\mathbb{R}$ for $i=m+1,\ldots,m+p$ and note that $(\mathcal{P}_D[v])^{(i)}=\mathcal{P}_{D^{(i)}}[v^{(i)}]$ for all $v\in\mathbb{R}^{m+p}$ and all $i$. In this way, 
\begin{align*}
	\mathcal{P}_{D^{(i)}}[u^{(i)}+\tfrac{1}{L(\Phi(S))}\nabla^{(i)}d(u)]&=\begin{cases} 0,\qquad & \mbox{if } i\in {\mathcal A}(u),\\
	u^{(i)}+\tfrac{1}{L(\Phi(S))}\nabla^{(i)}d(u), & \mbox{if } i \in {\mathcal I}(u). \end{cases}
\end{align*}
Also, since $\Phi(S)\supseteq S$, we have $L(\Phi(S))\ge L(S)$ and thus $\mathcal{P}_D[u+\tfrac{1}{L(\Phi(S))}\nabla d(u)]\in \Phi(S)$. It follows from \eqref{eq:ddgduu>=-L2uu} that
\begin{align}
&d^\star-d(u)\ge d(\mathcal{P}_D[u+\tfrac{1}{L(\Phi(S))}\nabla d(u)])-d(u)\nonumber\displaybreak[0]\\
&\ge\nabla d(u)^T\Bigl(\mathcal{P}_D[u+\tfrac{1}{L(\Phi(S))}\nabla d(u)]-u\Bigr)-\frac{L(\Phi(S))}{2}\Bigl\|u-\mathcal{P}_D[u+\tfrac{1}{L(\Phi(S))}\nabla d(u)]\Bigr\|^2.\label{eq:d>=dgdPu1LgduL2uPu1Lgd}
\end{align}
We now look at the right-hand side of \eqref{eq:d>=dgdPu1LgduL2uPu1Lgd}. For each $i\in{\mathcal A}(u)$,
\begin{align*}
&\nabla^{(i)}d(u)\Bigl(\mathcal{P}_{D^{(i)}}[u^{(i)}+\tfrac{1}{L(\Phi(S))}\nabla^{(i)}d(u)]-u^{(i)}\Bigr)\\
-&\frac{L(\Phi(S))}{2}\Bigl(u^{(i)}-\mathcal{P}_{D^{(i)}}[u^{(i)}+\tfrac{1}{L(\Phi(S))}\nabla^{(i)}d(u)]\Bigr)^2=-\nabla^{(i)}d(u)u^{(i)}-\frac{L(\Phi(S))}{2}\Bigl(u^{(i)}\Bigr)^2\displaybreak[0]\\
\ge&-\nabla^{(i)}d(u)u^{(i)}-\frac{L(\Phi(S))}{2}u^{(i)}\Bigl(-\tfrac{1}{L(\Phi(S))}\nabla^{(i)}d(u)\Bigr)=-\frac{1}{2}\nabla^{(i)}d(u)u^{(i)},
\end{align*}
where the inequality is due to $0\le u^{(i)}<-\tfrac{1}{L(\Phi(S))}\nabla^{(i)}d(u)$ $\forall i\in{\mathcal A}(u)$. For each $i\in{\mathcal I}(u)$, 
\begin{align*}
&\nabla^{(i)}d(u)\Bigl(\mathcal{P}_{D^{(i)}}[u^{(i)}+\tfrac{1}{L(\Phi(S))}\nabla^{(i)}d(u)]-u^{(i)}\Bigr)\\
-&\frac{L(\Phi(S))}{2}\Bigl(u^{(i)}-\mathcal{P}_{D^{(i)}}[u^{(i)}+\tfrac{1}{L(\Phi(S))}\nabla^{(i)}d(u)]\Bigr)^2\displaybreak[0]\\
=&\nabla^{(i)}d(u)\cdot\tfrac{1}{L(\Phi(S))}\nabla^{(i)}d(u)-\frac{L(\Phi(S))}{2}\Bigl(\tfrac{1}{L(\Phi(S))}\nabla^{(i)}d(u)\Bigr)^2=\frac{1}{2L(\Phi(S))}\Bigl(\nabla^{(i)}d(u)\Bigr)^2.
\end{align*}
Thus, \eqref{eq:d>=dgdPu1LgduL2uPu1Lgd} gives
\begin{align}
d^\star-d(u)&\ge\frac{1}{2}\Bigl(\sum_{i\in\mathcal{A}(u)}-\nabla^{(i)}d(u)u^{(i)}\Bigr)+\frac{1}{2L(\Phi(S))}\sum_{i\in\mathcal{I}(u)}\bigl(\nabla^{(i)}d(u)\bigr)^2.\label{eq:dd>=12sum-ndu+12Lsumnd}
\end{align}
This lead to
\begin{align}
\sum_{i\in\mathcal{A}(u)}-\nabla^{(i)}d(u)u^{(i)}\le2(d^\star-d(u)).\label{eq:sum-ndu<=2dd}
\end{align}
Moreover, notice that
\begin{align*}
\sum_{i\in\mathcal{I}(u)}\bigl(\nabla^{(i)}d(u)\bigr)^2&\ge\frac{1}{|\mathcal{I}(u)|}\Bigl(\sum_{i\in\mathcal{I}(u)}|\nabla^{(i)}d(u)|\Bigr)^2\ge\frac{1}{m+p}\Bigl(\frac{1}{\|u\|_\infty}\sum_{i\in\mathcal{I}(u)}-\nabla^{(i)}d(u)u^{(i)}\Bigr)^2.
\end{align*}
Also note that for each $i\in\mathcal{A}(u)$, $\nabla^{(i)}d(u)<0$ and thus $-\nabla^{(i)}d(u)u^{(i)}\ge0$. It then follows from \eqref{eq:dd>=12sum-ndu+12Lsumnd} that
\begin{align*}
\sum_{i\in\mathcal{I}(u)}-\nabla^{(i)}d(u)u^{(i)}\le\|u\|_\infty\sqrt{2L(\Phi(S))(m+p)(d^\star-d(u))}.
\end{align*}
Due to this, \eqref{eq:sum-ndu<=2dd}, and \eqref{eq:ff<=-gdu}, we conclude that \eqref{eq:ff<=usqrt2Lmp2sqrtdddd} holds.

Next, we derive a lower bound on $f(\bar{x}(u))-f^\star$. Since $\mathcal{L}(\bar{x}(u),u^\star)\ge\mathcal{L}(\bar{x}(u^\star),u^\star)=f^\star$,
\begin{align}
f(\bar{x}(u))-f^\star&\ge-\Bigl(\sum_{i=1}^mu^{\star(i)}g^{(i)}(\bar{x}(u))\Bigr)-(u^{\star(m+1:m+p)})^T(A\bar{x}(u)+b)\nonumber\displaybreak[0]\\
&\ge-\Bigl(\sum_{i=1}^mu^{\star(i)}\max\{0,g^{(i)}(\bar{x}(u))\}\Bigr)-(u^{\star(m+1:m+p)})^T(A\bar{x}(u)+b)\nonumber\displaybreak[0]\\
&\ge-\|u^\star\|\Delta(\bar{x}(u)).\label{eq:ff>=-usqrtDelta}
\end{align}
To bound $\Delta(\bar{x}(u))$, note from \eqref{eq:generalgradd} that if $g^{(i)}(\bar{x}(u))>0$, then $i\in{\mathcal I}(u)$. In addition, $\{m+1,\ldots,m+p\}\subseteq{\mathcal I}(u)$. Thus, $\bigl(\Delta(\bar{x}(u))\bigr)^2\le\sum_{i\in{\mathcal I}(u)}\bigl(\nabla^{(i)}d(u)\bigr)^2$.
It follows from \eqref{eq:dd>=12sum-ndu+12Lsumnd} that \eqref{eq:Delta<=2Ldd} holds. Due to \eqref{eq:ff>=-usqrtDelta} and \eqref{eq:Delta<=2Ldd} , we obtain \eqref{eq:ff>=-usqrt2Ldd}.

\subsection{Proof of Corollary~\ref{cor:primalfunctionvalueboundedsubgrad}}\label{ssec:proofcorprimalfunctionvalueboundedsubgrad}

Since for any compact $S\subset D$, $L(\Phi(S))\le\tilde{L}$, the corollary follows from Theorem~\ref{thm:primalfunctionvalue}.

\subsection{Proof of Corollary~\ref{cor:xx<=uulinear}}\label{ssec:proofcorxx<=uulinear}
Let $u,v\in\mathbb{R}^{m+p}$. Note that for problem~\eqref{eq:linearprimalprob}, \eqref{eq:tnftnfxx<=sumugugxx} is equivalent to
\begin{align*}
\Bigl(\tilde{\nabla}f(\bar{x}(u))-\tilde{\nabla}f(\bar{x}(v))\Bigr)^T(\bar{x}(u)-\bar{x}(v))\le(u-v)^T\tilde{A}(\bar{x}(v)-\bar{x}(u)).
\end{align*}
Additionally, similar to \eqref{eq:uuAxx<=lambdamaxxxuu}, we obtain
\begin{align*}
(u-v)^T\tilde{A}(\bar{x}(v)-\bar{x}(u))\le\sigma_{\max}(\tilde{A})\|\bar{x}(u)-\bar{x}(v)\|\cdot\|u-v\|.
\end{align*} 
Then, it follows from the strong convexity of $f$ on $X$ that $\|\bar{x}(u)-\bar{x}(v)\|\le\frac{\sigma_{\max}(\tilde{A})}{\theta}\|u-v\|$.

\subsection{Proof of Proposition~\ref{pro:primalfunctionvaluelinear}}\label{ssec:proofproprimalfunctionvaluelinear}

Let $u\in D$ and $u^\star\in D^\star$. We first prove that \eqref{eq:Delta<=2sigmathetadd} is satisfied. Note that $\bigl(\Delta(\bar{x}(u))\bigr)^2\le\|(\tilde{A}\bar{x}(u)+\tilde{b})-(\tilde{A}x^\star+\tilde{b})\|^2=\|\nabla d(u)-\nabla d(u^\star)\|^2$. Further, due to the inequalities $d(u)-d(u^\star)-\nabla d(u^\star)^T(u-u^\star)\le-\frac{\theta}{2\sigma_{\max}^2(\tilde{A})}\|\nabla d(u)-\nabla d(u)^\star\|^2$ \cite[Theorem~2.1.5]{Nesterov04} and $\nabla d(u^\star)^T(u-u^\star)\le0$, we obtain
\begin{align}
\|\nabla d(u)-\nabla d(u^\star)\|^2&\le-\frac{2\sigma_{\max}^2(\tilde{A})}{\theta}(d(u)-d(u^\star)-\nabla d(u^\star)^T(u-u^\star))\nonumber\displaybreak[0]\\
&\le\frac{2\sigma_{\max}^2(\tilde{A})}{\theta}(d(u^\star)-d(u)).\label{eq:ndnd<=2sigmathetadd}
\end{align}
Consequently, \eqref{eq:Delta<=2sigmathetadd} holds. From \eqref{eq:ff<=-gdu}, we have $f(\bar{x}(u))-f^\star\le-\nabla d(u)^Tu+\nabla d(u^\star)^T(u-u^\star)$. This, along with $\nabla d(u^\star)^Tu^\star=0$ \cite[Prop. 6.1.1]{Bertsekas03}, implies that
\begin{align*}
f(\bar{x}(u))-f^\star \le (\nabla d(u^\star)-\nabla d(u))^Tu\le\|u\|\cdot\|\nabla d(u^\star)-\nabla d(u)\|.
\end{align*}
It follows from \eqref{eq:ndnd<=2sigmathetadd}, \eqref{eq:ff>=-usqrtDelta}, and \eqref{eq:Delta<=2sigmathetadd} that \eqref{eq:ff<=u2lambdamaxAAthetadd} holds.

\subsection{Proof of Lemma~\ref{lem:dualgradientLipschitzlargerset}}\label{ssec:prooflemdualgradientLipschitzlargerset}

Let $u\in D$ and $v\in\tilde{D}$. We first show that
\begin{align}
&\|\bar{x}(u)-\bar{x}(v)\|\le\frac{\sqrt{m+1}}{\theta}\|u-v\|\nonumber\displaybreak[0]\\
&\quad\cdot\max\Bigl\{\sigma_{\max}(A),\max_{i\in\{1,\ldots,m\}}\|\nabla g^{(i)}(\bar{x}(u))\|,\max_{i\in\{1,\ldots,m\}}\|\nabla g^{(i)}(\bar{x}(v))\|\Bigr\}.\label{eq:xx<=uulargerset}
\end{align}

To prove \eqref{eq:xx<=uulargerset}, consider two mutually exclusive and exhaustive cases: 

Case (i) $v\in D$. In this case, from \eqref{eq:xx<=uuloose}, we obtain \eqref{eq:xx<=uulargerset}. 

Case (ii) $v\in\tilde{D}-D$. Let $I^-=\{i\in\{1,2,\ldots,m\}:\tilde{u}^{(i)}\le v^{(i)}<0\}\neq\emptyset$ and $w=\mathcal{P}_D[v]$, \emph{i.e.}, $w^{(i)}=0$ $\forall i\in I^-$ and $w^{(i)}=v^{(i)}$ $\forall i\in\{1,2,\ldots,m+p\}\backslash I^-$. Notice from the proof of Lemma~\ref{lem:xx<=cuu} that under Assumption~\ref{asm:inequalityconstraints}, \eqref{eq:xx<=uutight} holds with $u=w$, which gives
\begin{align}
&\|\bar{x}(w)-\bar{x}(v)\|\nonumber\\
\le&\frac{1}{\theta}\Bigl(\sum_{i=1}^m\|\nabla g^{(i)}(\bar{x}(v))\|\cdot|w^{(i)}-v^{(i)}|+\sigma_{\max}(A)\|w^{(m+1:m+p)}-v^{(m+1:m+p)}\|\Bigr)\nonumber\\
=&\frac{1}{\theta}\Bigl(\sum_{i\in I^-}-\|\nabla g^{(i)}(\bar{x}(v))\|v^{(i)}\Bigr).\label{eq:|xbu3xbu2|}
\end{align}
Due again to \eqref{eq:xx<=uutight}, we obtain
\begin{align}
&\|\bar{x}(u)-\bar{x}(w)\|\nonumber\\
\le&\frac{1}{\theta}\Bigl(\sum_{i=1}^m\|\nabla g^{(i)}(\bar{x}(u))\|\cdot|u^{(i)}-w^{(i)}|+\sigma_{\max}(A)\|u^{(m+1:m+p)}-w^{(m+1:m+p)}\|\Bigr)\nonumber\\
=&\frac{1}{\theta}\Bigl(\sum_{i\in I^-}\|\nabla g^{(i)}(\bar{x}(u))\|u^{(i)}+\!\!\!\!\!\!\!\!\!\!\!\!\!\!\sum_{i\in\{1,2,\ldots,m\}\backslash I^-}\!\!\!\!\!\!\!\!\!\!\!\!\!\!\|\nabla g^{(i)}(\bar{x}(u))\|\cdot|u^{(i)}-v^{(i)}|\nonumber\displaybreak[0]\\
&+\sigma_{\max}(A)\|u^{(m+1:m+p)}-v^{(m+1:m+p)}\|\Bigr).\label{eq:|xbu1xbu3|}
\end{align}
From \eqref{eq:|xbu3xbu2|}, \eqref{eq:|xbu1xbu3|}, and the fact that $|u^{(i)}-v^{(i)}|=u^{(i)}-v^{(i)}$ $\forall i\in I^-$, we have
\begin{align*}
&\|\bar{x}(u)-\bar{x}(v)\|\le\|\bar{x}(u)-\bar{x}(w)\|+\|\bar{x}(w)-\bar{x}(v)\|\nonumber\displaybreak[0]\\
\le&\frac{1}{\theta}\Bigl(\sum_{i=1}^m\max\Bigl\{\|\nabla g^{(i)}(\bar{x}(u))\|,\|\nabla g^{(i)}(\bar{x}(v))\|\Bigr\}|u^{(i)}-v^{(i)}|\nonumber\displaybreak[0]\\
&+\sigma_{\max}(A)\|u^{(m+1:m+p)}-v^{(m+1:m+p)}\|\Bigr).
\end{align*}
Akin to the last part of the proof of Lemma~\ref{lem:xx<=cuu}, it can be shown that \eqref{eq:xx<=uulargerset} holds for this case. 
Having proved \eqref{eq:xx<=uulargerset}, we now show that \eqref{eq:gdgd<=Luulargerset} holds. Due to Assumption~\ref{asm:inequalityconstraints}, Lemma~\ref{lem:xbounded} and \eqref{eq:generalgradd} still hold when $D$ is replaced by $\tilde{D}$. Thus, using the proof of Proposition~\ref{pro:gradientdualLipschitz} and \eqref{eq:xx<=uulargerset}, we have \eqref{eq:gdgd<=Luulargerset}.

\subsection{Proof of Lemma~\ref{lem:dualgradientLipschitzineq}}\label{ssec:prooflemdualgradientLipschitzineq}

Let $u\in S$ and $v\in\Psi(S)$. Since $\Psi(S)$ is convex and $S\subseteq\Psi(S)$, $u+\tau(v-u)\in\Psi(S)$ $\forall \tau\in[0,1]$. Then, from Lemma~\ref{lem:dualgradientLipschitzlargerset}, 
\begin{align*}
\|\nabla d(u+\tau(v-u))-\nabla d(u)\|\le\tau L(\Psi(S))\|u-v\|,\quad\forall \tau\in[0,1]. 
\end{align*}
It then follows from the proof of \cite[Lemma 1.2.3]{Nesterov04} that \eqref{eq:ddgduu>=-M2uu} holds. 

Next, we prove \eqref{eq:ddgduu<=-M2eta1etagdgd} and \eqref{eq:gdgduu<=-2M2eta1etagdgd}. Let $u,v\in S$. Note that if $\eta(S)=0$, then $\nabla d(u)=\nabla d(v)$. Thus, due to the concavity of $d$, \eqref{eq:ddgduu<=-M2eta1etagdgd} and \eqref{eq:gdgduu<=-2M2eta1etagdgd} hold. Now assume $\eta(S)>0$. To prove \eqref{eq:ddgduu<=-M2eta1etagdgd}, we utilize the idea from the proof of \cite[Theorem~2.1.5]{Nesterov04}. We define a function $\phi:\Psi(S)\rightarrow\mathbb{R}$ such that $\phi(w)=d(w)-\nabla d(u)^Tw$. Note that $\Phi$ is concave and $\nabla\Phi(u)=0$, which implies $\phi(u)\ge\phi(w)$ $\forall w\in\Psi(S)$. In addition, $v+\frac{1}{\eta}\nabla\phi(v)=v+\frac{1}{\eta}(\nabla d(v)-\nabla d(u))\in\Psi(S)$ $\forall\eta\ge\eta(S)$. It follows from \eqref{eq:ddgduu>=-M2uu} that
\begin{align*}
&\phi(u)\ge\phi(v+\frac{1}{\eta}\nabla\phi(v))\\
&=d(v+\frac{1}{\eta}\nabla\phi(v))-d(v)-\nabla d(v)^T\frac{1}{\eta}\nabla\phi(v)+\phi(v)+(\nabla d(v)-\nabla d(u))^T\frac{1}{\eta}\nabla\phi(v)\\
&\ge-\frac{L(\Psi(S))}{2}\|\frac{1}{\eta}\nabla\phi(v)\|^2+\phi(v)+\frac{1}{\eta}\|\nabla\phi(v)\|^2,\quad\forall\eta\ge\eta(S),
\end{align*}
which is equivalent to \eqref{eq:ddgduu<=-M2eta1etagdgd}. Moreover, \eqref{eq:gdgduu<=-2M2eta1etagdgd} can be obtained from \eqref{eq:ddgduu<=-M2eta1etagdgd} by interchanging $u$ and $v$ and adding the two inequalities.

\subsection{Proof of Theorem~\ref{thm:dualprimalconvrate}}\label{ssec:proofthmdualprimalconvrate}

Let $u^\star\in D^\star$ and $\alpha$ satisfy \eqref{eq:stepsize}. We first prove that $u_k\in D_0$ $\forall k\ge0$ by induction. Clearly, $u_0\in D_0$. Suppose $u_k\in D_0$ for some $k\ge0$. Then, for any $\eta>0$ such that $\eta\ge\eta(D_0)$,
\begin{align*}
\|u_{k+1}-u^\star\|^2&=\|\mathcal{P}_D[u_k+\alpha\nabla d(u_k)]-\mathcal{P}_D[u^\star+\alpha\nabla d(u^\star)]\|^2\\
&\le\|u_k+\alpha\nabla d(u_k)-u^\star-\alpha\nabla d(u^\star)\|^2\displaybreak[0]\\
&=\|u_k-u^\star\|^2+2\alpha(\nabla d(u_k)-\nabla d(u^\star))^T(u_k-u^\star)+\alpha^2\|\nabla d(u_k)-\nabla d(u^\star)\|^2\displaybreak[0]\\
&\le\Bigl(4\alpha\Bigl(\frac{L(\Psi(D_0))}{2\eta^2}-\frac{1}{\eta}\Bigr)+\alpha^2\Bigr)\|\nabla d(u_k)-\nabla d(u^\star)\|^2+\|u_k-u^\star\|^2,
\end{align*}
where the last inequality is due to \eqref{eq:gdgduu<=-2M2eta1etagdgd} and $u^\star\in D_0$. Notice that over the set $\{\eta>0:\eta\ge\eta(D_0)\}$, if $L(\Psi(D_0))>\eta(D_0)$, then $4\Bigl(\frac{L(\Psi(D_0))}{2\eta^2}-\frac{1}{\eta}\Bigr)$ achieves its minimum $-\frac{2}{L(\Psi(D_0))}$ at $\eta=L(\Psi(D_0))$; otherwise, $4\Bigl(\frac{L(\Psi(D_0))}{2\eta^2}-\frac{1}{\eta}\Bigr)$ reaches the minimum $4\Bigl(\frac{L(\Psi(D_0))}{2\eta(D_0)^2}-\frac{1}{\eta(D_0)}\Bigr)$ at $\eta=\eta(D_0)>0$. Hence, \eqref{eq:stepsize} leads to $4\alpha\Bigl(\frac{L(\Psi(D_0))}{2\eta^2}-\frac{1}{\eta}\Bigr)+\alpha^2<0$. Therefore, $\|u_{k+1}-u^\star\|\le\|u_k-u^\star\|$, which means that $u_{k+1}\in D_0$. This completes the proof by induction. With this property, we now prove \eqref{eq:gpmdualconvrate}. Because of Proposition~\ref{pro:gradientdualLipschitz} and because $D_0\subseteq\Psi(D_0)$, $\nabla d$ satisfies a Lipschitz condition on $D_0$ with Lipschitz constant no more than $L(\Psi(D_0))$. It then follows from the proof of \cite[Theorem~5.1]{Levitin66} that
\begin{align*}
&d(u_k)-d(u_{k+1})\le-\delta\|u_{k+1}-u_k\|^2,\\
&d^\star-d(u_k)\le\sqrt{\rho}\|u_{k+1}-u_k\|.
\end{align*}
Also, from \eqref{eq:stepsize}, we have $\alpha<\frac{2}{L(\Psi(D_0))}$ and thus $\delta>0$. Consequently, 
\begin{align*}
d^\star-d(u_{k+1})\le d^\star-d(u_k)-\frac{\delta}{\rho}(d^\star-d(u_k))^2.
\end{align*}
Then, from \cite[Lemma~6, Sec.~2.2]{Polyak87}, \eqref{eq:gpmdualconvrate} holds. Also, \eqref{eq:gpmdualconvrate} and \eqref{eq:xx<=sqrtddtheta} give \eqref{eq:gpmprimalconvdist}. Moreover, note that $\|u_k\|\le\|u^\star\|+\|u_0-u^\star\|$ $\forall k\ge0$. Then, \eqref{eq:gpmprimalconvratefunc} holds due to this, \eqref{eq:ff<=usqrt2Lmp2sqrtdddd}, and \eqref{eq:gpmdualconvrate}. Furthermore, \eqref{eq:gpmprimalconvratefunclb} comes from \eqref{eq:gpmdualconvrate} and \eqref{eq:ff>=-usqrt2Ldd}. Finally, \eqref{eq:gpmdualconvrate} and \eqref{eq:Delta<=2Ldd} yield \eqref{eq:gpminfconvrate}.

\subsection{Proof of Proposition~\ref{pro:fgdualprimalconvrate}}\label{ssec:proofprofgdualprimalconvrate}
Inequality \eqref{eq:fgdualconvrate} is derived in the proof of \cite[Corrolary~1]{Tseng08}, which, along with \eqref{eq:xx<=sqrtddtheta} in Theorem~\ref{thm:xx<=cuu}, gives \eqref{eq:fgprimalconvratedist}. In addition, from Corollary~\ref{cor:primalfunctionvalueboundedsubgrad}, we obtain \eqref{eq:fgprimalconvratefunc}, \eqref{eq:fgprimalconvratefunclb}, and \eqref{eq:fginfconvrate}.

\subsection{Proof of Proposition~\ref{pro:fgdualprimalconvrateineq}}\label{ssec:proofprofgdualprimalconvrateineq}
From \cite[Lemma~1]{Nedic09b}, for any $u\in\{u'\ge0:d(u')\ge d(\bar{u})\}$,
\begin{align*}
\|u\|_\infty\le\|u\|\le\frac{d(\bar{u})-f(\tilde{x})}{\max_{i\in\{1,2,\ldots,m\}}g^{(i)}(\tilde{x})}.
\end{align*}
Also, due to \eqref{eq:fgdualconvrate}, we have $d(u_k)\ge d(\bar{u})$ when $k+1\ge\Bigl(\frac{4\tilde{L}Q(u^\star,w_0)}{d^\star-d(\bar{u})}\Bigr)^{1/2}$ and $k\ge1$. It then follows from \eqref{eq:fgprimalconvratefunc} that \eqref{eq:fgprimalconvratefuncineq} holds.

\subsection{Proof of Proposition~\ref{pro:fgdualprimalconvratewholdspace}}\label{ssec:proofprofgdualprimalconvratewholdspace}
Inequality \eqref{eq:fgdualconvratewholespace} comes from \cite[Corollary~2]{Tseng08} and Proposition~\ref{pro:primalfunctionvaluelinear}. Moreover, because of \eqref{eq:xx<=sqrtddtheta}, \eqref{eq:ff<=u2lambdamaxAAthetadd}, and \eqref{eq:Delta<=2sigmathetadd}, we obtain \eqref{eq:fgprimalconvratedistwholespace}, \eqref{eq:fgprimalconvratefuncwholespace}, and \eqref{eq:fginfconvratewholespace}. Similar to the proof of Proposition~\ref{pro:fgdualprimalconvrateineq}, we have $d(u_k)\ge d(\bar{u})$ if $k+1\ge\sigma_{\max}(\tilde{A})\|u_0-u^\star\|\left(\frac{2\theta^{-1}}{d^\star-d(\bar{u})}\right)^{1/2}$ and $k\ge1$, implying that \eqref{eq:fgprimalconvratefuncineqwholespace} holds.

\bibliographystyle{IEEEtran}
\bibliography{reference2}

\begin{thebibliography}{10}
\providecommand{\url}[1]{#1}
\csname url@samestyle\endcsname
\providecommand{\newblock}{\relax}
\providecommand{\bibinfo}[2]{#2}
\providecommand{\BIBentrySTDinterwordspacing}{\spaceskip=0pt\relax}
\providecommand{\BIBentryALTinterwordstretchfactor}{4}
\providecommand{\BIBentryALTinterwordspacing}{\spaceskip=\fontdimen2\font plus
\BIBentryALTinterwordstretchfactor\fontdimen3\font minus
  \fontdimen4\font\relax}
\providecommand{\BIBforeignlanguage}[2]{{%
\expandafter\ifx\csname l@#1\endcsname\relax
\typeout{** WARNING: IEEEtran.bst: No hyphenation pattern has been}%
\typeout{** loaded for the language `#1'. Using the pattern for}%
\typeout{** the default language instead.}%
\else
\language=\csname l@#1\endcsname
\fi
#2}}
\providecommand{\BIBdecl}{\relax}
\BIBdecl

\bibitem{Lasdon70}
L.~S. Lasdon, \emph{Optimization Theory For Large Systems}.\hskip 1em plus
  0.5em minus 0.4em\relax New York, NY: Macmillan, 1970.

\bibitem{Bertsekas99}
D.~P. Bertsekas, \emph{Nonlinear Programming}.\hskip 1em plus 0.5em minus
  0.4em\relax Belmont, MA: Athena Scientific, 1999.

\bibitem{Nesterov04}
Y.~Nesterov, \emph{Introductory lectures on Convex Optimization: A Basic
  Course}.\hskip 1em plus 0.5em minus 0.4em\relax Norwell, MA: Kluwer Academic
  Publishers, 2004.

\bibitem{Low99}
S.~H. Low and D.~E. Lapsley, ``Optimization flow control, i: Basic algorithm
  and convergence,'' \emph{IEEE/ACM Transactions on Networking}, vol.~7, no.~6,
  pp. 861--874, 1999.

\bibitem{XiaoL04b}
L.~Xiao, M.~Johansson, and S.~Boyd, ``Simultaneous routing and resource
  allocation via dual decomposition,'' \emph{IEEE Transactions on
  Communications}, vol.~52, no.~7, pp. 1136--1144, 2004.

\bibitem{Giselsson13}
P.~Giselsson, M.~D. Doan, T.~Keviczky, B.~Schutter, and A.~Rantzer,
  ``Accelerated gradient methods and dual decomposition in distributed model
  predictive control,'' \emph{Automatica}, vol.~49, no.~3, pp. 829--833, 2013.

\bibitem{Nedic10b}
A.~Nedi\'{c} and A.~Ozdaglar, ``Cooperative distributed multi-agent
  optimization,'' in \emph{Convex Optimization in Signal Processing and
  Communications}, Y.~Eldar and D.~Palomar, Eds.\hskip 1em plus 0.5em minus
  0.4em\relax Cambridge University Press, 2010, pp. 340--386.

\bibitem{Larsson99}
T.~Larsson, M.~Patriksson, and A.-B. Str\"{o}mberg, ``Ergodic, primal
  convergence in dual subgradient schemes for convex programming,''
  \emph{Mathematical Programming}, vol.~86, no.~2, pp. 283--312, 1999.

\bibitem{Nedic09b}
A.~Nedi\'{c} and A.~Ozdaglar, ``Approximate primal solutions and rate analysis
  for dual subgradient methods,'' \emph{SIAM Journal on Optimization}, vol.~19,
  no.~4, pp. 1757--1780, 2009.

\bibitem{Devolder12}
O.~Devolder, F.~Glineur, and Y.~Nesterov, ``Double smoothing technique for
  large-scale linearly constrained convex optimization,'' \emph{SIAM Journal on
  Optimization}, vol.~22, no.~2, pp. 702--727, 2012.

\bibitem{Patrinos13}
P.~Patrinos and A.~Bemporad, ``An accelerated dual gradient-projection
  algorithm for embedded linear model predictive control,'' \emph{IEEE
  Transactions on Automatic Control}, vol.~59, no.~1, pp. 18 -- 33, 2013.

\bibitem{Beck14}
A.~Beck, A.~Nedi\'{c}, A.~Ozdaglar, and M.~Teboulle, ``An $\uppercase{O}(1/k)$
  gradient method for network resource allocation problems,'' \emph{IEEE
  Transactions on Control of Network Systems}, vol.~1, no.~1, pp. 64--73, 2014.

\bibitem{Necoara14}
I.~Necoara and V.~Nedelcu, ``Rate analysis of inexact dual first-order methods
  application to dual decomposition,'' \emph{IEEE Transactions on Automatic
  Control}, vol.~59, no.~5, pp. 1232--1243, 2014.

\bibitem{Nedelcu14}
V.~Nedelcu, I.~Necoara, and Q.~Tran-Dinh, ``Computational complexity of inexact
  gradient augmented lagrangian methods: Application to constrained mpc,''
  \emph{SIAM Journal on Control and Optimization}, vol.~52, no.~5, pp.
  3109--3134, 2014.

\bibitem{Tseng08}
P.~Tseng, ``On accelerated proximal gradient methods for convex-concave
  optimization,'' submitted to {\em SIAM Journal on Optimization}.

\bibitem{Beck09}
A.~Beck and M.~Teboulle, ``A fast iterative shrinkage-thresholding algorithm
  for linear inverse problems,'' \emph{SIAM Journal on Imaging Sciences},
  vol.~2, no.~1, pp. 183--202, 2009.

\bibitem{Bertsekas03}
D.~P. Bertsekas, A.~Nedi\'{c}, and A.~Ozdaglar, \emph{Convex Analysis and
  Optimization}.\hskip 1em plus 0.5em minus 0.4em\relax Belmont, MA: Athena
  Scientific, 2003.

\bibitem{Levitin66}
E.~S. Levitin and B.~T. Polyak, ``Constrained minimization problems,''
  \emph{USSR Computational Mathematics and Mathematical Physics}, vol.~6, pp.
  1--50, 1966, english version in Zh. Vychisl. Mat. mat. Fiz., vol.6,
  pp.787-823, 1966.

\bibitem{Urruty96}
J.~B.~H. Urruty and C.~Lemar\'{e}chal, \emph{Convex Analysis and Minimization
  Algorithms}.\hskip 1em plus 0.5em minus 0.4em\relax Germany: Springer, 1996.

\bibitem{Polyak87}
B.~T. Polyak, \emph{Introduction to Optimization}.\hskip 1em plus 0.5em minus
  0.4em\relax New York, NY: Optimization Software, Inc., 1987.

\end{thebibliography}

\end{document}